\documentclass[12pt,twoside]{article}

\usepackage[T1]{fontenc}
\usepackage[latin1]{inputenc}
\usepackage[english]{babel}
\usepackage[babel]{csquotes}

\usepackage{cite}

\usepackage{amssymb}
\usepackage{amsmath}
\usepackage{amsthm}
\usepackage{latexsym}
\usepackage{graphicx}
\usepackage{mathrsfs}
\usepackage{mathrsfs}

\usepackage[usenames,dvipsnames]{color}

\definecolor{viola}{rgb}{0.3,0,0.7}
\definecolor{lilla}{rgb}{0.8,0,0.8}
\definecolor{ciclamino}{rgb}{0.5,0,0.5}
\definecolor{blu}{rgb}{0,0,0.7}
\definecolor{verde}{rgb}{0,0.5,0.2}
\definecolor{rosso}{rgb}{0.8,0,0}

\def\piecol #1{{\color{verde}#1}}
\def\luca #1{{\color{blu}#1}}
\def\pier #1{{\color{rosso}#1}}
\def\ele#1{{\color{viola}#1}}
\def\pcol #1{{\color{rosso}#1}}
\newcommand\gius[2]{{\color{lilla}#1}} 
\def\pieluc #1{{\color{rosso}#1}}
\def\piecol #1{#1}
\def\luca #1{#1}
\def\pier #1{#1}
\def\ele #1{#1}
\def\pcol #1{#1}
\renewcommand\gius[2]{{#1}}
\def\pieluc #1{#1}

\theoremstyle{plain}
\newtheorem{thm}{Theorem}[section]
\newtheorem{lem}[thm]{Lemma}

\theoremstyle{definition}
\newtheorem{rmk}[thm]{Remark}

\def\En{\mathbb{N}}

\def\Ar{\mathbb{R}}

\def\Pi{\mathbb{P}}

\def\cL{\mathscr{L}}

\def\beq{\begin{equation}}
\def\eeq{\end{equation}}

\def\rarr{\rightarrow}

\def\l|{\left\|}
\def\r|{\right\|}

\def\L2{L^2(\Omega)}
\def\H1{H^1(\Omega)}

\def\rarrw{\rightharpoonup}
\def\weakstar{\stackrel{*}{\rightharpoonup}}
\def\embed{\hookrightarrow}

\begin{document}
\begin{center}
\piecol{%
{\huge\rm A doubly nonlinear Cahn--Hilliard system\\[0.2cm] 
with nonlinear viscosity\/\footnote{\piecol{{\bf 
Acknowledgments.}\quad\rm PC gratefully acknowledges some 
financial support from the MIUR-PRIN Grant \piecol{2015PA5MP7 ``Calculus of Variations''}, the GNAMPA (Gruppo Nazionale per l'Analisi Matematica, 
la Probabilit\`a e le loro Applicazioni) of INdAM (Istituto 
Nazionale di Alta Matematica) and the IMATI -- C.N.R. Pavia.}}}
\\[0.5cm]
{\large\sc Elena Bonetti}\\
{\normalsize e-mail: {\tt elena.bonetti@unimi.it}}\\[.25cm]
{\large\sc Pierluigi Colli}\\
{\normalsize e-mail: {\tt pierluigi.colli@unipv.it}}\\[.25cm]
{\large\sc Luca Scarpa}\\
{\normalsize e-mail: {\tt luca.scarpa.15@ucl.ac.uk}}\\[.25cm]
{\large\sc Giuseppe Tomassetti}\\
{\normalsize e-mail: {\tt giuseppe.tomassetti@uniroma3.it}}\\[.25cm]
$^{(1)}$
{\small Dipartimento di Matematica ``F.Enriques'', Universit\`a degli 
Studi di Milano}\\ 
{\small Via Saldini 50, 20133 Milano, Italy}\\[.2cm]
$^{(2)}$
{\small Dipartimento di Matematica ``F. Casorati'', Universit\`a di Pavia}\\
{\small Via Ferrata 1, 27100 Pavia, Italy}\\[.2cm]
$^{(3)}$
{\small Department of Mathematics, University College London}\\
{\small Gower Street, London WC1E 6BT, United Kingdom}\\[.2cm]
$^{(4)}$
{\small Dipartimento di Ingegneria - Sezione Ingegneria Civile}\\
{\small Universit\`a degli Studi ``Roma Tre'', Via Vito Volterra 62, Roma, Italy}
}

\end{center}
       
\begin{abstract}
\vskip-1cm
In this paper we discuss a family of viscous Cahn--Hilliard equations
with a non-smooth viscosity term. This system may be viewed as an approximation 
of a ``forward-backward'' parabolic equation. \ele{The resulting \pcol{problem} 
is highly nonlinear, coupling in the same equation two nonlinearities 
with the diffusion term}. In particular, we prove 
existence of solutions for the related initial and boundary value problem. 
\ele{Under suitable assumptions, we} also \ele{state uniqueness and} continuous dependence on data. 
\\[.5cm]
\end{abstract}

\pagestyle{myheadings}
\newcommand\testopari{\sc \pier{A doubly nonlinear Cahn--Hilliard system 
with nonlinear viscosity}}
\newcommand\testodispari{\sc \pier{Bonetti -- Colli -- Scarpa -- Tomassetti}}
\markboth{\testodispari}{\testopari}


\thispagestyle{empty}

\section{Introduction}
\setcounter{equation}{0}
\label{intro}
 In this paper we are concerned with the following variant of the viscous Cahn--Hilliard system 
\begin{align}
  \label{eq1}
  \partial_t u - \Delta\mu = 0 \qquad&\text{in } \Omega\times(0,T)\,,\\
  \label{eq2}
  \mu \in \partial_t u + \beta(\partial_t u) - \Delta u +\psi'(u)+ g \qquad&\text{in } \Omega\times(0,T)\,,\\
  \label{boundary}
  \partial_{\bf n} u =0\,, \quad \mu=0 \qquad&\text{in } \partial\Omega\times(0,T)\,,\\
  \label{init}
  u(0)=u_0 \qquad&\text{in } \Omega\,\gius{,}{.}
\end{align}
which governs the evolution of \pcol{the} \emph{density} $u$ and \pcol{the} \emph{chemical potential} $\mu$ of a species diffusing in a space--time region $\Omega\times(0,T)$\ele{,} with $\Omega$ a domain of the three--dimensional physical space. Here $\psi$ is \pcol{a} possibly non--convex free energy density\pcol{,} $\beta:\mathbb R\to 2^{\mathbb R}$ is a maximal monotone graph \ele{such that $0\in\beta(0)$}, \pcol{and $g$ is a datum.}

The model we consider differs from the standard Cahn--Hilliard \pcol{one} in that the \pcol{usual} prescription $\mu=-\Delta u+\psi(u)$ for chemical potential is replaced by the inclusion \eqref{eq2}, whose right--hand side \pcol{also} depends on the time derivative $\partial_t u$, \ele{due to the presence of viscosity contributions: a linear and a nonlinear viscosity terms}, the latter \pcol{represented by the} maximal monotone graph $\beta$. 

\pcol{We point out that the usual boundary condition for $\mu$ that is 
considered for the Cahn--Hilliard system is the Neumann homogeneous boundary 
condition. \pieluc{To this concern, we mention the contribution \cite{MS} in which a system
similar to \eqref{eq1}--\eqref{init} has been treated in the case of Neumann homogeneous boundary conditions for $\mu$ (see the later Remark~\ref{pierrem}), while \cite{MZ} is involved with the analysis of an analogous system in which the viscosity term in the equation corresponding to \eqref{eq2} has the form
$\partial_t \alpha (u)$, with $\alpha$ maximal monotone graph.}
Here, also in view of the argumentation developed in \cite{BCT}, 
we follow a different approach by prescribing a {\it Dirichlet condition} 
for $\mu $ on the boundary. To this concern, let us underline that despite 
the second equality in \eqref{boundary}, in our approach we can treat the case of a {\it non-homogeneous} Dirichlet boundary condition for the chemical potential. Indeed, in this case it would suffice to consider a shifted chemical potential $\mu$ by subtracting the harmonic extension of the boundary datum and by incorporating it into the known term $g$ in \eqref{eq2}.}

\pcol{Concerning} the nonlinear viscosity \ele{contribution $\beta(\pcol{\partial_t u})$}, if it is removed, then we recover the viscous Cahn--Hilliard system proposed by Novick-Cohen in \cite{Novic1988viscous} and studied by Novick--Cohen and Pego in \cite{NovicP1991TAMS}. If, on the other hand, we remove from the right--hand side of \eqref{eq2} \ele{the diffusive contribution $-\Delta u$, which actually can be read as a space regularization}, we obtain a play--type hysteresis model with dynamical effects that has been proposed to model macroscopical hysteretic effects in diffusion processes\pcol{:} see for instance \cite{schweizer2012} for applications to flow in porous media and \cite{Tomas,BCT} for applications in solid--state hydrogen storage. \ele{In \pcol{particular}, in \cite{BCT} we state an existence and uniqueness result for a \pcol{system similar to \eqref{eq1}-\eqref{init} but} where we neglect the contribution of $-\Delta u$.} Other possibilities to model hysteresis associated to diffusion have been considered in \cite{botkinetal2016}, a paper \pcol{containing numerical results as well.}

\ele{The PDE system we propose in this paper aims} to model processes where phase separation due to diffusion is accompanied by hysteretic behavior. As a possible application of this model, we suggest hydrogen adsorption in solid--state storage devices, where the major manifestation of hysteresis is in the fact that the pressure that is needed to induce hydrogen adsorption is higher than the pressure needed to induce desorption~\cite{Latroce}.  


\section{Thermodynamical consistency}
The starting point of Cahn--Hilliard type models is the \emph{diffusant--balance equation}
\begin{equation}\label{eq:11}
  \begin{aligned}
  &\partial_t u + \operatorname{div}{\pcol{\boldsymbol{h}}}=0
\end{aligned}
\end{equation}
where $u$ is the mass density and $\boldsymbol h$ is the mass flux, related to the chemical potential $\mu$ through \emph{Fick's law}
\begin{equation}\label{eq:2}
  \boldsymbol h=-M\nabla\mu,
\end{equation}
with $M$ a positive scalar mobility. In the standard Cahn--Hilliard model the system \eqref{eq:11}--\eqref{eq:2} is \ele{complemented} by the equation
\begin{equation}\label{eq:6}
\mu=\psi'(u)-\alpha\Delta u,
\end{equation}
whose right--hand side is the variational derivative of the \emph{free energy contained in $\Omega$}:
\begin{equation}\label{eq:18}
  \mathcal F(u)=\int_\Omega \pcol{\Big(} \psi(u)+\frac \alpha 2|\nabla u|^2 \pcol{\Big)}.
\end{equation}
\pcol{This} section is dedicated to showing that the replacement of \eqref{eq:6} with the 
\pcol{differential inclusion}
\begin{equation}\label{eq:21}
  \mu\in{\pcol{\sigma}}\partial_t u+\beta(\partial_t u)-\alpha\Delta u+\psi'(u)+g
\end{equation}
leads to a thermodynamically consistent model, provided that ${\pcol{\sigma}}\ge 0$ and $\beta:\mathbb R\to 2^{\mathbb R}$ is a maximal monotone graph \pcol{with $0\in \beta(0)$}. \gius{}{It is quite instructive to notice that the system that results from \eqref{eq:22} and \eqref{eq:21} can be written formally as a gradient flow
\begin{equation}
  \partial\mathcal R(\partial_t u)+D_u\mathcal F(u)\in 0,
\end{equation}
of the free energy \eqref{eq:18} with respect to the dissipation potential
\begin{equation}
  \mathcal R(v)=\int_\Omega \frac 1 {2 M}\nabla (\Delta^{-1}w)\cdot \nabla(\Delta^{-1} w)+\frac {\pcol{\sigma}}2 {v^2}+B(v),
\end{equation}
where $B(\cdot)$ is the (proper, lower semicontinuous and convex) potential of the maximal monotone graph $\beta(\cdot)$. However, we prefer to interpret \eqref{eq2} a constitutive prescription  which finds its justification within the framework proposed by Gurtin in \cite{Gurtin96} to derive thermodynamically--consistent models of Allen--Cahn and Cahn--Hilliard type.} 

As a start, we recall that, within Gurtin's framework, the set--valued function
\begin{equation}\label{eq:14}
  \mathcal I(R)=-\int_{\partial R} \mu\,\boldsymbol h\cdot\boldsymbol n_R,\qquad R\subset\Omega,
\end{equation}
is the inflow of chemical energy due to mass transport into any \emph{fixed} smooth subregion $R\subset\Omega$ (whose outward unit normal at the boundary we denote by $\boldsymbol n_R$) of the domain where the diffusion process takes place. 

Besides transport of the diffusant, there are other mechanisms by  which $R$ can exchange energy with its environment. Central to Gurtin's approach is the presumption that energy--exchange processes other than mass transport take place through \emph{expenditure of power} associated with time variations of $u$, the field representing the degrees of freedom of the physical system that is being modeled; in particular, Gurtin postulates that there exists a vector field $\boldsymbol s$ and scalar field $g$ (both defined on $\Omega$) such that the power expended by the environment on an arbitrary part $R$ admits the following representation:
\begin{equation}\label{eq:10}
  \mathcal W(R)=\int_{\partial R} (\boldsymbol s\cdot\boldsymbol n_R)\,\partial_t u+\int_R g\, \partial_t u.
\end{equation}
By way of analogy with continuum mechanics, the scalar field $g$ is interpreted as a \emph{microscopic force} that external agents exert on the part $R$, whereas the vector field $\boldsymbol s$ is interpreted as a \emph{microscopic stress}. Along with this interpretation goes the postulate that the following \emph{microforce balance} holds in $\Omega$:
\begin{equation}\label{eq:7}
  \operatorname{div}\boldsymbol s+f+g=0,
\end{equation}
where $g$ is an internal microforce accounting for the interaction of the medium with the diffusant. 

Inasmuch as the diffusant--balance equation \eqref{eq:11} is the antecedent of the elliptic equation
\begin{equation}\label{eq:22}
  \partial_t u \pcol{{}-{}} M\Delta\mu=0,
\end{equation}
which is obtained from  \eqref{eq:11} through the constitutive equation \eqref{eq:2} for the flux $\boldsymbol h$, the microscopic force balance \eqref{eq:7} is the antecedent of the inclusion \eqref{eq:21} needed to close \eqref{eq:22}. This inclusion is obtained by providing suitable \emph{constitutive prescriptions} for $\boldsymbol s$, $f$, $g$. We note on passing that when the microforce balance \eqref{eq:7} is combined with the constitutive prescriptions 
\begin{equation}\label{eq:19}
  \boldsymbol s=\alpha\nabla u,\qquad g=0,
\end{equation}
along with 
\begin{equation}\label{eq:1}
f=\mu-\psi'(u),
\end{equation}
then the standard prescription \eqref{eq:6} for the chemical potential is recovered. 

When willing to generalize \eqref{eq:6}, a procedure to filter out thermodynamical consistent prescriptions is based on the requirement that all possible evolution processes (which in principle can be driven by tuning the external microscopic force $g$) be consistent with the Second Law of Thermodynamics. In isothermal conditions, the Second Law reduces to a dissipation principle which dictates that the time derivative of the free energy contained in any part $R$ of the domain $\Omega$ must not exceed the rate at which energy, both in the form of chemical inflow and microscopic power expenditure, is supplied to that part, or equivalently that the dissipation rate within $R$ be non negative: 
\begin{equation}\label{eq:15}
 \mathcal D(R):= \mathcal I(R)+\mathcal W(R)-\frac{d}{dt}\int_R \pcol{\Big(}\psi(u)+\frac \alpha 2|\nabla u|^2\pcol{\Big)}\ge 0.
\end{equation}
To see how this requirement imposes restrictions on the microscopic stress $\boldsymbol s$ and on the microscopic force $f$, we begin by noting that the expression of the inflow of chemical energy $\mathcal I(R)$ and the expended power $\mathcal W(R)$ \pcol{expressed} in \eqref{eq:14} and \eqref{eq:10} can be given the form of volume integrals by making use of the mass balance equation \eqref{eq:11} and the microforce balance equation \eqref{eq:7}, which yield
\begin{equation}\label{eq:16}
  \mathcal I(R)=\int_R \mu\, \partial_t u-\boldsymbol h\cdot\nabla\mu,\qquad
  \mathcal W(R)=\int_{R} \boldsymbol s\cdot\nabla\partial_t u- f \partial_t u.
\end{equation}
Then, the combination of \eqref{eq:15} with \eqref{eq:16} and \eqref{eq:18}, followed by a standard localization argument\pcol{,} shows that dissipation rate \eqref{eq:15} is a measure, and thermodynamical consistency is ensured provided that the density $\delta$ of this measure is non-negative: 
\begin{equation}\label{eq:3}
\delta=(\boldsymbol s-\alpha\nabla\partial_tu)\cdot\nabla\partial_t u+(\mu-f-\psi'(u))\partial_tu-\pcol{\boldsymbol{h}}\cdot\nabla\mu\ge 0.
\end{equation}
It is easy to see that the choices \eqref{eq:2} and \eqref{eq:19} together with \eqref{eq:1} guarantee that \eqref{eq:3} is never violated during any evolution process, since they entail that the dissipation--rate density is $\delta=M|\nabla\mu|^2$, which is non--negative  if the mobility $M$ is positive.

The choice that leads to \eqref{eq:21}, still consistent with \eqref{eq:3}, is to keep \eqref{eq:19} and \eqref{eq:2}, but replace the equation \eqref{eq:1} with the following inclusion:
\begin{equation}\label{eq:4}
 f\in\mu-\psi'(u)-{\pcol{\sigma}}\partial_t u-\beta(\partial_t u),
\end{equation}
where ${\pcol{\sigma}}\ge 0$ and $\beta:\mathbb R\to 2^{\mathbb R}$ is a maximal monotone graph. By combining \eqref{eq:4} with \eqref{eq:19} and \eqref{eq:7} we obtain \eqref{eq:21}. In this case, the dissipation--rate density is
\begin{equation}\label{eq:20}
  \delta=M|\nabla\mu|^2+({\pcol{\sigma}}\partial_t u)^2+\beta(\partial_tu)\partial_t u\pcol{{}\geq 0.}
\end{equation}
\ele{The last inequality  exploits the monotonicity of $\beta$ and the fact that $0\in\beta(0)$. For the sake of completeness, let us quote also the theory introduced by Fr\'emond \cite{fre}
in the framework of phase transitions, where the systems are recovered in a variational setting \pier{by} introducing a generalization of the principle of virtual powers including microscopic forces and motions. In this setting thermodynamical consistency is ensured once the dissipation functional is a pseudo-potential of dissipation \`a la Moreau, i.e. a convex, \pcol{lower semicontinuous}, non-negative functional, which takes value zero for zero dissipation.}
\ele{%
\begin{rmk}\label{rem_flux}
{Let us point out that} the Cahn--Hilliard system can be formally rewritten as a gradient flow (in a suitable variational framework)
\begin{equation*}
  D\mathcal H(\partial_tu)+D\mathcal F(u)=0,
\end{equation*}
with respect to the free energy \eqref{eq:18} and the dissipation potential
\begin{equation*}
  \mathcal H(v)=\frac 1 {2M}\int_\Omega \nabla(\gius{G}{\Delta^{-1}}v)\cdot\nabla(\gius{G}{\Delta^{-1}}v)\gius{,}{.}
\end{equation*}
\gius{where}{Here} \gius{$G=-\Delta^{-1}$}{} denotes the \gius{inverse}{solution operator} of the \gius{negative}{} Laplacian with homogeneous Dirichlet boundary conditions \gius{(see the definition at the beginning of Section 4 below)}{}.
\gius{It is instructive to notice that, analogously,}{ nalogously, it is quite instructive to notice that} the system that results from \eqref{eq:22} and \eqref{eq:21} can be written formally as a gradient flow
\begin{equation*}
  \partial\mathcal R(\partial_t u)+D_u\mathcal F(u)\in 0,
\end{equation*}
with the choice of  the free energy \eqref{eq:18}  and the dissipation potential
\begin{equation*}
  \mathcal R(v)=\gius{\mathcal H(v)}{\int_\Omega \frac 1 {2 M}\nabla (\gius{G}{\Delta^{-1}}w)\cdot \nabla(\gius{G}{\Delta^{-1}} w)}+\gius{\int_\Omega
  \pcol{\Big(}
  \frac {\pcol{\sigma}}2 \, {v^2}+\widehat{\beta}(v)}{}\pcol{\Big)},
\end{equation*}
where \gius{$\widehat{\beta}$}{} is the (proper, lower semicontinuous and convex) potential of the maximal monotone graph \gius{$\beta$}{$\beta(\cdot)$}. The variational structure we have put in evidence \pcol{leads to an energy type estimate for the solutions of the system, which is recovered in our mathematical analysis.} However, we \pcol{underlined the interpretation of} \eqref{eq2} as a constitutive prescription  which finds its justification within the framework proposed by Gurtin in \cite{Gurtin96} to derive thermodynamically--consistent models of Allen--Cahn and Cahn--Hilliard type.
\end{rmk}%
}


\section{Setting and main results}

\setcounter{equation}{0}
\label{results}

In this section, we \ele{first} describe the general setting of the work and \ele{our} notation. \ele{Then,}
 the problem that we are dealing with and the main results are stated.

Throughout the paper, $\Omega$ \ele{denotes} a smooth bounded domain in $\Ar^3$ \ele{with  boundary} $\Gamma$
and $T>0$ is a fixed final time; for any $t\in(0,T]$ we use the notation
\[
  Q_t:=\Omega\times(0,t)\,, \quad \Sigma_t:=\Gamma\times(0,t)\,, \quad Q:=Q_T\,, \quad \Sigma:=\Sigma_T\ele{.}
\]
Moreover, we introduce the spaces
\[
  H:=L^2(\Omega)\,, \quad V:=H^1(\Omega)\,,
\]
endowed with their usual norms $\l|\cdot\r|_H$ and $\l|\cdot\r|_V$, respectively, and we identify $H$ with its dual,
so that $(V, H, V^*)$ is a Hilbert triplet. \piecol{We also introduce the subspace 
$V_0 := H^1_0 (\Omega)$ of $V$} \pier{and the space 
$$ W= \left\{y\in H^2(\Omega): \ \partial_{\bf n}y=0 \, \  \text{\pier{a.e. on} } \Gamma\right\}. 
$$
The} symbol $\left<\cdot, \cdot\right>$
is used to denote the duality pairing between $V^*$ and $V$, while $(\cdot, \cdot)$ is the usual scalar product of $H$.

Let us now make some rigorous assumptions on the data. We assume that
\begin{gather}
  \label{psi1}
  \psi:(a, b)\rarr\Ar\,, \quad -\infty\leq a<b\leq+\infty\,,\\
  \label{psi2}
  \psi\in C^2(a,b)\,,\\
  \label{psi3}
  \psi(r)\geq0 \quad\forall\, r\in(a,b)\,,\\
  \label{psi4}
  \lim_{r\rarr a^+}\psi'(r)=-\infty\,, \quad \lim_{r\rarr b^-}\psi'(r)=+\infty\,,\\
  \label{psi5}
  \psi''(r)\geq -K \quad\forall\,r\in(a,b)
\end{gather}
for a positive constant $K$. Furthermore, as far as $\beta$ is concerned, we assume that
\beq
  \label{beta}
  \beta:\Ar\rarr2^{\Ar} \quad\text{is maximal monotone}\,, \quad \beta(0)\ni0\,,
\eeq
so that \pier{we can introduce} the proper, convex, lower semicontinuous function
\beq
  \label{beta_hat}
  \pier{\widehat{\beta}}:\Ar\rarr[0,+\infty] \quad\text{such that}\quad\pier{\widehat{\beta}}(0)=0\,, \quad\partial\pier{\widehat{\beta}}=\beta\,,
\eeq
\ele{where the subdifferential is intended in the sense of convex analysis}.
Finally, we make the following hypotheses on $g$ and the intial datum:
\begin{gather}
  \label{g}
  g\in H^1(0,T; H)\cap L^2(0,T; V)\,,\\[0.2cm]
  u_0 \in \pier{W} ,  \, \ \psi'(u_0)\in H , \ \pier{\hbox{ and there exist }  a_0>a, \ \, b_0 < b\,} 
  \nonumber \\
\hskip3cm \pier{\hbox{ such that }\,  a_0\leq u_0 (x) \leq b_0 \quad \hbox{for all }\,  x\in\Omega}\,. \label{u_0}
\end{gather}

\begin{rmk}
\ele{Let us note that, for the sake of clarity,} in \eqref{u_0} we are \pier{asking that the composition of $\psi'$ 
with $u_0$ belongs to the space $H$ \ele{even if, actually,} this is ensured by the other conditions in 
\eqref{u_0}. \ele{Indeed, the} Sobolev embedding theorems,  $H^2(\Omega) \subset C^0(\overline{\Omega}) $
and the bounds $a_0\leq u_0\leq b_0 $ imply that $\psi(u_0)$, $\psi'(u_0)$, $\psi''(u_0)$, all lie in $L^\infty (\Omega)$.}
\end{rmk}
\ele{Now, it remains} to \ele{introduce} the initial values \pier{$\mu_0\in V_0 \cap H^2(\Omega)$ and $u_0'\in H$}. To this end,
the natural \pier{approach} is to require that $(\mu_0, u_0, u_0')$ satisfy the elliptic system
induced by \eqref{eq1}--\eqref{eq2} at the initial time:
\beq
  \label{system_init}
  \begin{cases}
  u_0'-\Delta\mu_0=0 \quad \pier{\hbox{a.e. in }\,\Omega}, \\[0.1cm]
  \mu_0\in u_0'+\beta(u_0')-\Delta u_0+\psi'(u_0)+g(0) \quad\pier{\hbox{a.e. in }\,\Omega}.
  \end{cases}
\eeq
Setting $z_0:=-\Delta u_0 + \psi'(u_0) + g(0)$, \ele{from \eqref{system_init} it follows}
$u_0'\in(I+\beta)^{-1}(\mu_0-z_0)$ and
$$(I+\beta)^{-1}(\mu_0-z_0)-\Delta\mu_0{\ele \ni}0 \quad\pier{\hbox{a.e. in }\,\Omega.}$$ 
Since $z_0\in H$ thanks to \eqref{g}--\eqref{u_0} and
$(I+\beta)^{-1}$ is \pier{monotone and} Lipschitz continuous, the last \ele{inclusion is actually an equation and it} admits a solution \pier{$\mu_0\in V_0\cap H^2(\Omega)$,}
which is a posteriori unique \pier{due to the strong monotonicity of the operator $-\Delta$ with domain $V_0 \cap H^2(\Omega)$.} It is clear that\pier{, by} defining $u_0':=(I-\beta)^{-1}(\mu_0-z_0)\in H$, then \pier{the pair 
$(\mu_0, u_0')$ solves} the system \eqref{system_init}.

We are now ready to state the main results of the work.
\begin{thm}
  \label{thm1}
  In the current setting, assume that
  \beq
    \label{growth_psi} \luca{(a,b) =\Ar\,, \quad 
    \exists\, M>0: \quad |\psi''(r)|\leq M\left(1+|r|^5\right) \quad\forall\, r\in 
    \Ar} \,.
  \eeq
  Then there exists a triplet $(u, \mu, \xi)$ such that  
\begin{gather}
    \label{u}
    u\in W^{1,\infty}(0,T; H)\cap H^1(0,T; V)\cap L^\infty(0,T; \pier{W})\\
    \label{mu}
    \mu\in L^\infty(0,T; V_0\cap H^2(\Omega))\cap L^2(0,T; H^3(\Omega))\,,\\
    \label{xi_psi}
   \xi \in L^\infty(0,T; H)\,, \quad \ \psi'(u)\in \pier{L^\infty(0,T; H)}\,,\\
    \label{incl}
    \xi\in\beta(\partial_t u) \quad\text{a.e.~in } Q\,,\\
    \label{1}
    \partial_t u(t)-\Delta \mu(t) = 0 \quad\text{for a.e.~}t\in(0,T)\,,\\
    \label{2}
    \mu(t)=\partial_t u(t)+\xi(t)-\Delta u(t) +\psi'(u(t)) + g(t)\quad\text{for a.e.~}t\in(0,T)\,,\\
    \label{3}
    u(0)=u_0.
  \end{gather}
\end{thm}

\begin{rmk}
\label{pier-rem}
\pier{Note that in the above formulation the boundary conditions \eqref{boundary} for $u$ and $\mu$  
are actually hidden in the regularity properties \eqref{u} and \eqref{mu},  
due to the definitions of the spaces $W$ and $V_0$. We also point out that, in the light of  the compact embedding $W\embed  C^0(\, \pier{\overline{\Omega}} \, )$ and, e.g., \cite[Cor.~4, p.~85]{simon}, $u $ is in 
$C^0([0,T]; C^0(\, \pier{\overline{\Omega}} \, ))$, then continuous everywhere in 
$\overline{Q}$; this fact,  along with \eqref{psi1}, \eqref{psi2} and $(a,b)= \Ar $, entails that
$\psi'(u)\in C^0(\,\pier{\overline{Q}} \,).$}
\end{rmk}

\begin{thm}
  \label{thm1bis} 
  In the current setting, assume that
  \beq
   \label{beta_sublinear}
   \piecol{D(\beta) = \Ar\,,} \quad
    \exists\, M>0: \quad |s|\leq M\left(1+ |r|\right) \quad\forall r\in\Ar\,,\; \forall\, s\in\beta(r)\,.
  \eeq
  Then there exists a triplet $(u, \mu, \xi)$ \pier{satisfying \eqref{u}--\eqref{3}.}
\end{thm}

\pier{The following \ele{result} states a continuous dependence of the solution upon the data in a particular situation.
Of course, in the case here considered also the uniqueness property follows.}

\begin{thm}
  \label{contdep} 
Let $(g_i, u_{0,i} )$, $i=1,2,$ be two sets of data satisfying \eqref{g}--\eqref{u_0} and let  $(u_i, \mu_i, \xi_i)$
denote the corresponding solutions. We assume that there exist  $ \overline a >a, \  \overline  b < b $ such that 
\beq \overline a \leq u_i (x,t) \leq \overline b \quad \hbox{for all }\,  (x,t) \in \overline Q , \quad i=1,2. \label{hyp-pier}
\eeq
Then, there exist a constant $C_d$, depending only on the data, such that 
  \begin{align}
   \label{dipcont}
    &\| \mu_1 -\mu_2 \|_{L^2(0,T;V_0)} +   \| u_1 - u_2 \|_{H^1(0,T;H) \cap L^\infty (0,T;V)} + \int_Q (\xi_1 - \xi_2 )
    (\partial_t u_1 - \partial_t u_2) \nonumber \\ 
    &\leq C_d \left(
  \| g_1 -g_2 \|_{L^2(0,T;H)} +  \| u_{0,1} - u_{0,2} \|_V \right). 
  \end{align}
\end{thm}

\pier{Note that in the framework of Theorem~\ref{thm1}, and more generally whenever $(a,b)=\Ar$,  
the condition \eqref{hyp-pier} is satisfied for all solutions 
(cf.~Remark~\ref{pier-rem}). \ele{Thus, under assumptions \eqref{growth_psi}, Theorems \ref{thm1} and \ref{contdep} actually ensure existence, uniqueness, and continuous dependence on data of the solution of our problem. This result cannot be extended to the case of the (weaker) assumptions of Theorem \ref{thm1bis}. Indeed, in} the setting of Theorem~\ref{thm1bis} it turns out that 
$u_1, u_2 \in C^0(\,\pier{\overline{Q}} \,)$ as well, but it is not clear whether \eqref{hyp-pier} holds true in this case.}

\pieluc{\begin{rmk}\label{pierrem}
We think it is worth to make a comparison between our results and the existence results proved in the paper~\cite{MS}, of which we have been aware only after a first version of our paper 
was accomplished. In \cite[Theorem~2.1]{MS} the authors show the existence of a weaker solution under the assumption~\eqref{beta_sublinear} and lighter conditions on the initial datum $u_0$, that is, $u_0 \in V$ and $ \psi(u_0) \in L^1 (\Omega).$ In \cite[Theorem~2.2]{MS} the assumption \eqref{beta_sublinear} is maintained but the 
hypotheses on the initial values are even stronger than \eqref{u_0}; on the other hand, 
they can avoid the coercive term $\partial_t u $ in \eqref{2}. The third existence result 
\cite[Theorem~2.4]{MS} can be compared with our Theorem~\ref{thm1}; however, it requires the condition \eqref{growth_psi} for $\psi'$ and not for $\psi''$, so that our assumption turns out more general. 
\end{rmk}}

\pier{We conclude this section by stating a general rule concerning the \ele{notation for} constants 
that appear in the estimates to be performed in the sequel.
The small-case symbol $c$ stands for a generic constant
whose values might change from line to line and even within the same line of the inequalities
and can depend only on~$\Omega$, $T$, on the data of our problem, 
and on the \ele{fixed} constants and the norms of the functions involved in the assumptions of our statements.}

\ele{For the sake of clarity, let us sketch the outline of the proof of our results. Due to the high nonlinear structure we have to approximate our system. 
\begin{itemize}
\item We first rewrite \eqref{eq2} in terms of the variable $u$. Then, we approximate the nonlinear contributions in the resulting equation  through the Yosida approximations of the involved maximal monotone graphs and the introduction of a truncation operator. To prove existence of a solution for the approximated system, we perform a fixed point argument. 
\item We establish uniform estimates on the approximated solutions, independently of the approximating parameter. Due to the variational structure of the evolution system, we are able to get the so-called ``energy estimate'' (see Remark \ref{rem_flux}). Then, we perform further refined estimates \pcol{which are important} to deal with the doubly nonlinear structure of the approximated version of \eqref{eq2}. These estimates are different, in the case of the assumptions of Theorem~\ref{thm1} or~\ref{thm1bis}. 
\item Due to the uniform bounds on the solutions, we can pass to the limit with respect to the vanishing approximating parameter by (weak-strong) compactness and lower semicontinuity results, identifying the limit as a solution to the original system. 
\item Finally, contracting estimates ensure the continuous \pcol{dependence} result stated by Theorem~\ref{contdep}. This \pcol{result} holds \pcol{in particular} under the assumptions of Theorem \ref{thm1bis}, so that it completes the existence \pcol{analyis providing the} uniqueness of the solution.
\end{itemize}
}


\section{The approximated problem}
\setcounter{equation}{0}
\label{Approx}

In this section, we approximate the original problem and show
existence and uniqueness of the approximated solutions.
To this end, we need some preparatory work.

First of all, we introduce the operator $G$ on $H$ as the inverse of the \pier{Laplacian}
with \pier{homogeneous} Dirichlet boundary conditions: more precisely, for any $f\in H$, we set $Gf$ as the unique solution $y$ to the system
\[
  \begin{cases}
  -\Delta y=f &\text{in } \Omega\,,\\
  y=0 &\text{on } \Gamma\,.
  \end{cases}
\]
It is well-known that $G:H\rarr H^2(\Omega)\cap H^1_0(\Omega)$ is well-defined and continuous.
With this notation, note that equation \eqref{eq1} together with the respective boundary condition on $\mu$ contained in
\eqref{boundary} can be rewritten in a compact form as
\[
  \mu=-G\partial_t u\,.
\]
Consequently, substituting this expression in \eqref{eq2},
we can write equation \eqref{eq2} only in terms of 
the variable $u$:
\[
  \partial_t u + G\partial_t u +\beta(\partial_t u)-\Delta u+\psi'(u)+g\ni0\quad\text{in } \Omega\times(0,T)\,.
\]

Secondly, we introduce the function
\beq
  \label{gamma}
  \gamma:(a,b)\rarr\Ar\,, \qquad \gamma(r):=\psi'(r)+Kr \quad\forall\, r\in\ele{(a,b)}
\eeq
and we note that the hypotheses \eqref{psi2} and \eqref{psi5} ensure that $\gamma$
is a maximal monotone graph in $\Ar\times\Ar$\pier{, as well as a $C^2$ increasing function with range all of $\Ar$}. As a consequence, there exists
$x_0\in(a,b)$ such that $0\in\gamma(x_0)$ and it is \pier{well defined}
the proper, convex and lower semicontinuous function 
\beq
  \label{gamma_hat}
  \widehat{\gamma}:\Ar\rarr[0,+\infty] \quad\text{such that}\quad \widehat{\gamma}(x_0)=0\,, \quad \partial\widehat{\gamma}=\gamma\,.
\eeq 
Moreover, we \pier{specify} the operators on $H$
\begin{gather}
  \label{a}
  A:=I+G+\beta\,, \quad D(A):=\left\{y\in H: \;\exists\,\xi\in H\,,\; \xi\in\beta(y) \text{ a.e.~in } \Omega\right\}\,,\\
  \label{b}
  B:=-\Delta+\gamma\,, \quad D(B):=\left\{y\in \pier{W: \;  \gamma(y)\in H}\right\}\,,
\end{gather}
where the symbol $I$ denotes the identity in $H$.
Taking these remarks into account, 
the \pier{system \eqref{eq1}--\eqref{init} takes the abstract form}
\begin{gather}
\pier{A(\partial_t u (t))+B(u(t) ) \ni Ku(t)-g(t) \quad \hbox{in }H, \, \hbox{ for a.e. }\, t\in (0,T),}\nonumber
\\ 
\pier{u(0)=u_0\quad\hbox{in } H.} \nonumber
\end{gather}

We are now ready to build the approximation.
For every $\lambda\in(0,1)$, we introduce the operators
\begin{gather}
  \label{a_approx}
  A_\lambda:=I+G+\beta_\lambda\,, \quad D(A_\lambda):=H\,,\\
  \label{b_approx}
  B_\lambda:=-\Delta+\gamma_\lambda+\lambda I\,, \quad D(B_\lambda):=\pier{W}\,,
\end{gather}
where $\beta_\lambda$ and $\gamma_\lambda$ are the Yosida approximations of $\beta$ and $\gamma$, respectively.
\pier{We point out that (see, e.g., \cite[Prop.~2.6, p.~28]{Brezis}) 
\begin{gather}
|\beta_\lambda (r)| \leq |\beta^\circ (r)| \quad \hbox{for all } r\in D(\beta) ,  \label{yos1} \\
|\gamma_\lambda (r)| \leq |\gamma (r)| \quad \hbox{for all } r\in (a,b) , \label{yos2}
\end{gather} 
where $\beta^\circ (r)$ is the element of $\beta (r)$ with minimal modulus; on the other hand, in view of \eqref{gamma} note that 
$\gamma$ is single-valued.} Furthermore, let $T_\lambda$ be the usual truncation operator at level $\pier{1/\lambda}$, i.e.
\beq
  \label{trunc}
  T_\lambda:\Ar\rarr\Ar\,, \qquad T_\lambda(r):=
  \begin{cases}
  \pier{1/\lambda} \quad&\text{if } r>\pier{1/\lambda}\,,\\
  r &\text{if } |r|\leq\pier{1/\lambda}\,,\\
  -\pier{1/\lambda} &\text{if } r<-\pier{1/\lambda}\,,
  \end{cases}
  \quad r\in\Ar\,.
\eeq
The approximated problem is the following:
\begin{gather}
  \label{approx1}
\pier{  A_\lambda(\partial_t u_\lambda(t)) + B_\lambda(u_\lambda(t)) = KT_\lambda(u_\lambda(t)) - g (t)\quad \hbox{in }H, \, \hbox{ for a.e. }\, t\in (0,T),}
\\ 
\pier{\quad u_\lambda(0)=u_0 \quad\hbox{in } H.} \label{approx2}
\end{gather}

The idea \ele{we follow} to prove existence of solutions for \eqref{approx1}--\eqref{approx2} is to use a fixed point argument in the following way:
we fix a suitable $v$ instead of $u_\lambda$ in the right-hand side of \eqref{approx1}--\eqref{approx2}, 
we solve the resulting equation using an existence result
for abstract doubly nonlinear parabolic equations and we prove then some estimates on the solutions allowing us
to apply \pier{the} \piecol{Schauder} fixed point theorem. Let us now go into the details.

We recall some properties of $A_\lambda$ and $B_\lambda$ in the following lemmata. \ele{More precisely, we need to point out that they are maximal monotone operators 
and that some coerciveness and continuity properties hold, with respect to suitable norms. In particular, the following results will enable us to refer to the theory of maximal monotone operators to ensure existence of solutions in the fixed point argument we are  introducing.}
\begin{lem}
  \label{lemma1}
  For any $\lambda\in(0,1)$, the operators $A_\lambda$ and $B_\lambda$ are maximal monotone on $H$.
  Moreover, the following conditions hold:
  \begin{align*}
  (i) \quad&\forall\,y\in H \quad \left(A_\lambda y, y\right)\geq\l|y\r|_H^2\,,\\
  (ii) \quad&\exists\,C_\lambda>0: \quad\forall \,y\in H \quad \l|A_\lambda y\r|_H\leq C_\lambda\l|y\r|_H\,,\\
  (iii) \quad &B_\lambda=\partial\phi_\lambda\,, \, \hbox{ \pier{with} } \,\phi_\lambda:H\rarr(-\infty,+\infty] \text{ proper, convex, lower semicontinuous}\,,\\
  (iv) \quad &D(\phi_\lambda)\subseteq V \quad\text{and}\quad\phi_\lambda(y)\geq\frac{\lambda}{2}\l|y\r|^2_V \quad\forall\, y\in V\,.
  \end{align*}
\end{lem}
\begin{proof}
  Firstly, $G$ is monotone (since so is $-\Delta$), linear and continuous, hence maximal monotone:
  this implies\pier{,} together with the Lipschitz continuity of $\beta_\lambda$ and $\gamma_\lambda$ that
  $A_\lambda$ and $B_\lambda$ are maximal monotone on $H$\pier{; concerning $B_\lambda$, see, e.g., \cite[Cor.~1.3, p.~48]{Barbu} or \cite[Lemme~2.4, p.34]{Brezis}}. 
  Secondly, $(i)$ and $(ii)$ are \pier{rather} obvious \pier{due to} the monotonicity of $G+\beta_\lambda$, \pier{the definition of $G$, the Lipschitz continuity of $\beta_\lambda$ and the fact that $\beta_\lambda (0)=0$.}
  Let us focus on $(iii)$: let $\widehat{\gamma_\lambda}:\Ar\rarr[0,+\infty)$ be the
  \pier{Moreau--Yosida} regularization of $\widehat{\gamma}$. 
  Then it is a standard matter to check that $(iii)$ is satisfied with
  the following choice:
  \beq
  \label{fi_lam}
  \phi_\lambda(\pier{v}):=\begin{cases}
  \frac12\int_\Omega|\nabla \pier{v}|^2 + \int_\Omega\widehat{\gamma_\lambda}(\pier{v}) 
  + \frac{\lambda}{2}\int_\Omega|\pier{v}|^2 \quad&\text{if } \pier{v}\in V\,,\\
  +\infty &\text{if } \pier{v}\notin V\,.
  \end{cases}
  \eeq
  Finally, it is clear that $(iv)$ holds by \pier{virtue of \eqref{fi_lam}} and the positivity of $\widehat{\gamma_\lambda}$.
\end{proof}

\begin{lem}
  \label{lemma2}
  The following conditions hold:
  \begin{align*}
  (v) \quad &A_\lambda=\partial h_\lambda\,,  \, \hbox{ \pier{with} } \,h_\lambda:H\rarr[0,+\infty] \text{ proper, convex, lower semicontinuous}\,,\\
\pier{(vi)} \quad &\pier{\hbox{the effective domain of the convex conjugate $h_\lambda^*$ is all of } H\,,}  \\
  \pier{(vii)}\quad &A_\lambda \text{ is bounded in } H \text{ (i.e.,~maps bounded sets in bounded sets)}\,,\\
  \pier{(viii)} \quad &B_\lambda:V\rarr V^* \text{ is strongly monotone and Lipschitz continuous}\,.
  \end{align*}
\end{lem}
\begin{proof}
  It is not difficult to check that a possible 
  choice of $h_\lambda$ is
  \beq
    \label{h_lam}
    h_\lambda(\pier{v}):=\frac12\int_\Omega|\pier{v}|^2+F^*(\pier{v})+\int_\Omega\pier{\widehat{\beta}_\lambda}(\pier{v})\,, \quad \pier{v}\in H\,,
  \eeq
  where \pier{$F^*$ denotes the conjugate of the proper, convex and lower semicontinuous function $F:H\rarr[0 +\infty]$} given by
  \[
    F(\pier{v}):=\begin{cases} \frac12\int_\Omega|\nabla \pier{v}|^2 \quad&\text{if }\pier{v}\in H^1_0(\Omega)\,,\\ +\infty &\text{otherwise}\,, \end{cases}
  \]
  and $\pier{\widehat{\beta}_\lambda}$ is the Moreau\pier{--Yosida} regularization of $\widehat{\beta}$, so that $(v)$ is proved.
  Secondly, since $G+\beta_\lambda$ is maximal monotone, the operator $A_\lambda$ is invertible and $A^{-1}_{\lambda}$ is \pier{well defined}.
  Moreover, for every $\pier{v},\pier{w}\in H$, by the monotonicity and the Lipschitz continuity of $G$ and $\beta_\lambda$, we have
  \[
  \l|\pier{v}-\pier{w}\r|_H^2\leq \left(A_\lambda \pier{v}-A_\lambda \pier{w}, \pier{v}-
  \pier{w}\right)\leq\left(1+\l|G\r|_{\cL(H,H)}+\frac{1}{\lambda}\right)\l|\pier{v}-
  \pier{w}\r|_H^2\,,
  \]
  so that $A_\lambda, A_\lambda^{-1}:H\rarr H$ are bi-Lipschitz continuous:
  this implies conditions $\pier{(vi)}$ and $\pier{(vii)}$.
  Finally, condition $\pier{(viii)}$ follows directly from the definition of $B_\lambda$ and the fact that $\gamma_\lambda$
  is monotone and Lipschitz continuous\pier{: of course, here the operator $B_\lambda$
has to be understood as the extension
$$\pier{\langle B_\lambda v, w \rangle = \int_\Omega \nabla v \cdot \nabla w + \int_\Omega(\gamma_\lambda (v) + \lambda v) w , \quad v,w \in V,}$$
this operator working from $V$ to $V^*$.}
\end{proof}

We are now ready to describe the fixed point argument. Fix $v\in L^2(0,T; H)$ and consider the problem
\begin{gather}
  \label{u_lamb1}
\pier{  A_\lambda(\partial_t u_\lambda(t)) + B_\lambda(u_\lambda(t)) = KT_\lambda(v(t)) - g (t)\quad \hbox{in }H, \, \hbox{ for a.e. }\, t\in (0,T),}
\\ 
\pier{\quad u_\lambda(0)=u_0 \quad\hbox{in } H.} \label{u_lamb2}
\end{gather}
By Lemma \ref{lemma1} and the hypotheses \eqref{g}--\eqref{u_0}, \pier{it turns out that} 
we can apply the existence result
contained in \cite[Thm.~2.1]{colli-visin}: we deduce that there exists
\beq
  \label{u_lam}
  u_\lambda \in H^1(0,T; H)\cap L^\infty(0,T; V)
\eeq
\pier{satisfying \eqref{u_lamb1}--\eqref{u_lamb2} and fulfilling}
\begin{gather}
  \label{app3}
  A_\lambda\left(\partial_t u_\lambda\right)\,, \;B_\lambda\left(u_\lambda\right) \;\in\; L^2(0,T; H)\,.
\end{gather}
It is natural now to state the following result, \ele{\pcol{which entails} a uniform bound on $u_\lambda$, in suitable norms.}
\begin{lem}
  \label{lemma3}
  If \pier{$v\in L^2(0,T; H)$ and{}} $u_\lambda$ satisfies \pier{\eqref{u_lam}--\eqref{app3} and \eqref{u_lamb1}--\eqref{u_lamb2}}, then 
  \beq
    \label{shaud1}
    \l|\partial_t u_{\lambda}\r|_{L^2(0,T; H)} + \l|u_{\lambda}\r|_{L^\infty(0,T; V)}
    \leq C_\lambda
  \eeq
  for \pier{some} positive constant $C_\lambda$ \pier{independent of $\| v \|_{L^2(0,T; H)} $}. Moreover, \pier{there is another constant $D_\lambda$ such that,} if
  $v_1, v_2\in L^2(0,T; H)$ and $(u_{\lambda, 1}, u_{\lambda, 2})$ are any 
  respective solutions to \pier{\eqref{u_lamb1}--\eqref{u_lamb2}}, then
  \beq
    \label{shaud2}
    \l|\partial_t(u_{\lambda, 1}-u_{\lambda, 2})\r|_{L^2(0,T; H)} + \l|u_{\lambda, 1}-u_{\lambda, 2}\r|_{L^\infty(0,T; V)}
    \leq \pier{D_\lambda} \l|v_1-v_2\r|_{L^2(0,T; H)}\,.
  \eeq
\end{lem}
\begin{proof}
  Let $u_\lambda$ be a solution to \pier{\eqref{u_lamb1}--\eqref{u_lamb2}}. Testing \pier{\eqref{u_lamb1}} by $\partial_t u_\lambda$, integrating
  on $[0,t]$ and taking into account the definitions of $A_\lambda$, $B_\lambda$ and $T_\lambda$,
  thanks to the Young inequality for every $t\in[0,T]$ we have
  \[
  \begin{split}
  \pier{\int_0^t\!\!\int_\Omega}&|\partial_t u_\lambda(s)|^2\,ds + 
  \pier{\int_0^t\!\!\int_\Omega}\left(G\partial_t u_\lambda(s) + \beta_\lambda(\partial_t u_\lambda(s))\right)\partial_t u_\lambda(s)\,ds\\
  &\qquad\qquad\qquad\qquad+\frac12\int_\Omega|\nabla u_\lambda(t)|^2 + 
  \int_\Omega\widehat{\gamma}(u_\lambda(t)) +
  \frac{\lambda}{2}\int_\Omega|u_\lambda(t)|^2\\
  &=\frac12\int_\Omega|\nabla u_0|^2 + 
  \int_\Omega\widehat{\gamma}(u_0) +
  \frac{\lambda}{2}\int_\Omega|u_0|^2 +
  \pier{\int_0^t\!\!\int_\Omega}\left(KT_\lambda \pier{( v(s))}-g(s)\right)\partial_t u_\lambda(s)\,ds\\
  &\leq\l|u_0\r|_V^2 + 
  \l|\widehat{\gamma}(u_0)\r|_{L^1(\Omega)} +
  \int_Q\left(\frac{K^2}{\lambda^2}+|g|^2\right) + \frac{1}{2}\pier{\int_0^t\!\!\int_\Omega}|\partial_t u_\lambda(s)|^2\,ds\,,
  \end{split}
  \]
  from which \eqref{shaud1} follows with \pier{obvious} choice of $C_\lambda$. \pier{Note indeed that $\int_\Omega\widehat{\gamma}(u_\lambda(t)) \geq 0 $ by \eqref{gamma_hat} and the second term on the left-hand side is non-negative by the monotonicity of $G$ and $\beta_\lambda$ and the facts that $G(0)=0$ and $\beta_\lambda (0)=0$.}
  Furthermore, testing the difference
  between \pier{\eqref{u_lamb1}} evaluated for $i=1,2$ by $\partial_t(u_{\lambda,1}-u_{\lambda,2})$,
  setting $v:=v_1-v_2$
   and $u_\lambda:=u_{\lambda, 1}-u_{\lambda, 2}$ for convenience, 
  by the monotonicity of $G+\beta_\lambda$
  it is easy to deduce
  \[
  \begin{split}
  &\pier{\int_0^t\!\!\int_\Omega}|\partial_t u_\lambda(s)|^2\,ds + \frac12\int_\Omega|\nabla u_\lambda(t)|^2 
  +\frac{\lambda}{2}\int_\Omega|u_\lambda(t)|^2\\
  &\quad\leq K\pier{\int_0^t\!\!\int_\Omega}(\pier{T_\lambda ( v_1)-T_\lambda( v_2)})(s)\partial_t u_\lambda(s)\,ds
  -\pier{\int_0^t\!\!\int_\Omega}\left(\gamma_{\lambda}(u_{\lambda, 1})-\gamma_\lambda(u_{\lambda, 2})\right)(s)\partial_t u_\lambda(s)\,ds\\
  &\quad\leq \frac{K}{\lambda}\pier{\int_0^t\!\!\int_\Omega}|v(s)|\,
  |\partial_t u_\lambda(s)| \,ds + 
  \frac{1}{\lambda}\pier{\int_0^t\!\!\int_\Omega}|u_\lambda(s)|\,
  |\partial_t u_\lambda(s)| \,ds\,.
  \end{split}
  \]
  By the already proved estimate \eqref{shaud1} and the Young inequality, we deduce
  \begin{align*}
  \frac12\pier{\int_0^t\!\!\int_\Omega}|\partial_t u_\lambda(s)|^2\,ds + \frac12\int_\Omega|\nabla u_\lambda(t)|^2 
  +\frac{\lambda}{2}\int_\Omega|u_\lambda(t)|^2 \\
  \leq \pier{\frac{K^2}{\lambda^2}\pier{\int_0^t\!\!\int_\Omega}|v(s)|^2 \,ds + 
  \frac{1}{\lambda^2}\pier{\int_0^t\!\!\int_\Omega}|u_\lambda(s)|^2 \,ds}
  \end{align*}
\pier{from which \eqref{shaud2} follows by applying the Gronwall lemma.}
\end{proof}

\ele{We are now in the position of detailing the fixed point argument.} 
It makes sense to introduce the set
\[
  X_\lambda:=\left\{v\in H^1(0,T; H)\cap L^\infty(0,T; V):\;
  \l|\partial_t v\r|_{L^2(0,T; H)} + \l|v\r|_{L^\infty(0,T; V)}
    \leq C_\lambda\right\}\,,
\]
where $C_\lambda$ is given as in \eqref{shaud1}.
By the second part of Lemma \ref{lemma3}, it is well-defined the map
\beq
  \label{shaud_map}
  \Xi: X_\lambda\rarr X_\lambda\,, \quad v\mapsto u_\lambda\,;
\eeq
since $X_\lambda$ is a closed convex set in $L^2(0,T; H)$ and $\Xi$ is continuous and compact with respect to the
topology of $L^2(0,T; H)$ by \eqref{shaud2}, Schauder's fixed point theorem ensures that there exists
\beq
  \label{u_app}
  u_\lambda \in H^1(0,T; H)\cap L^\infty(0,T; V)
\eeq
such that
\begin{gather}
   \label{1app}
  A_\lambda\left(\partial_t u_\lambda\right)\,, \;B_\lambda\left(u_\lambda\right) \;\in\; L^2(0,T; H)\,,\\
  \label{2app}
  A_\lambda\left(\partial_t u_\lambda(t)\right)+B_\lambda\left(u_\lambda(t)\right) = KT_\lambda \pier{(u_\lambda(t))}-g(t)
         \quad\text{for a.e.~}t\in(0,T)\,,\\
  \label{3app}
  u_\lambda(0)=u_0\,.
\end{gather}
Moreover, \pier{by arguing as in the proof of \eqref{shaud2}, the final application of the Gronwall lemma} implies that such $u_\lambda$ is also unique.

Finally, let us now deduce some regularity results for the approximated solutions: we collect them
in the following lemma.
\begin{lem}
  \label{approx_sol}
  In the current notation, setting $\mu_\lambda:=-G\partial_t u_\lambda$, we have
  \begin{gather}
    \label{u_lambda}
    u_\lambda \in \pier{C^1([0,T]; H)\cap H^1(0,T; V)  \cap C^0([0,T]; W) \cap L^2(0,T; H^3(\Omega))} \,,\\
    \label{mu_lambda}
      \mu_\lambda \in \pier{C^0([0,T]; V_0 \cap H^2(\Omega))\cap L^2(0,T; H^3(\Omega))}  
  \end{gather}
  and the following equations hold for every $t\in[0,T]$:
    \begin{gather}
    \label{eq1_approx}
      \partial_t u_\lambda(t)-\Delta\mu_\lambda(t)=0\,,\\[0.2cm]
      \mu_\lambda(t)=\partial_t u_\lambda(t) +\beta_\lambda(\partial_t u_\lambda(t))-
                         \Delta u_\lambda(t)  \hskip4cm \nonumber \\
                         \hskip4cm{}
                         +\gamma_\lambda(u_\lambda(t)) +\lambda u_\lambda(t) + g(t) -KT_\lambda \pier{(u_\lambda(t))}\,.
       \label{eq2_approx}                                               
  \end{gather}
\end{lem}
\begin{proof}
  First of all, thanks to \eqref{g}--\eqref{u_0} and Lemma \ref{lemma2}, we can apply the result contained in 
  \cite[Thm.~2.2]{colli-visin},
  so that \pier{besides the} uniqueness of $u_\lambda$ we deduce  \pier{that $u_\lambda\in H^1(0,T; V)$ as well as 
$$A_\lambda \partial_t u_\lambda \in H^1(0,T; V^*) \cap L^\infty(0,T; H) .$$}%
Secondly, by definition of $G$ and the regularity that we have just proved for $u_\lambda$,
  we have  $\Delta\mu_\lambda=\partial_tu_\lambda\in L^2(0,T; V)$: hence, by elliptic regularity we also \pier{infer that}
  $\mu_\lambda \in L^2(0,T; \pier{V_0 \cap H^3(\Omega)})$.
  Moreover, with the information that we have so far, by \eqref{2app} and the definitions of $A_\lambda$ and $B_\lambda$,
  equation \eqref{eq2_approx} holds for almost every $t\in(0,T)$: here, 
  note that all the terms except $-\Delta u_\lambda$ are in $L^2(0,T; V)$, so that by difference and elliptic regularity, 
  we recover that $u_\lambda \in L^2(0,T; \pier{W\cap H^3(\Omega)})$. As a consequence,
  we have that $u_\lambda \in \pier{H^1(0,T; V) \cap L^2(0,T; \pier{W\cap H^3(\Omega)})}  \subseteq C^0([0,T];\pier{W})$.
  Furthermore, if we set 
  $z_\lambda:=-\Delta u_\lambda + \gamma_\lambda(u_\lambda) +\lambda u_\lambda+g-KT_\lambda \pier{(u_\lambda)}$,
  then it is now clear that $z_\lambda\in C^0([0,T]; H)$; since 
  $\partial_t u_\lambda=A_\lambda^{-1}z_\lambda$ and
  $A_\lambda^{-1}:H\rarr H$ is Lipschitz continuous, we recover also that
  $u_\lambda \in C^1([0,T]; H)$, so that \eqref{u_lambda} \pier{follows}.
  Finally, by definition of $\mu_\lambda$ and elliptic regularity we deduce that $\mu_\lambda \in C^0([0,T]; \pier{V_0 \cap H^2(\Omega)})$,
  and \eqref{mu_lambda} is proved. It is now immediate to check now that \eqref{eq1_approx}--\eqref{eq2_approx} hold
  for every $t\in[0,T]$.
\end{proof}


\section{Uniform estimates}
\setcounter{equation}{0}
\label{estimates}

This section is devoted to proving some uniform estimates on the approximated solutions,
independently of $\lambda \pier{{}\in (0,1)}$. Let us point out that the first two estimates do not 
involve the different hypoyheses \eqref{growth_psi} and \eqref{beta_sublinear}; on the other side,
the third estimate has to be performed only when assuming \eqref{growth_psi} and the fourth
when assuming \eqref{beta_sublinear}. \pier{Henceforth, for the integrals over space and time, it occurs that we use the shorter notation
$\int_{Q_t}$ in place of $\pier{\int_0^t \! \int_\Omega}\,$, for $t\in [0,T].$}

\subsection{The first estimate}

We \pier{sum and subtract $u_\lambda$ in \eqref{eq2_approx}, then} test equation \eqref{eq1_approx} by $\mu_\lambda$ \pier{and equation} \eqref{eq2_approx} by $\partial_t u_\lambda$\pier{. Next, we} take the difference
and integrate on $[0,t]$ for every $t\in[0,T]$: if $\widehat{\gamma_\lambda}:\Ar\rarr[0,+\infty)$ is the Moreau\pier{--Yosida}
regularization of $\widehat{\gamma}$, using the Young inequality we \pier{infer that}
\[
  \begin{split}
  &\int_{Q_t}|\partial_t u_\lambda|^2 + \int_{Q_t}|\nabla\mu_\lambda|^2 + \int_{Q_t}\beta_\lambda(\partial_t u_\lambda)\partial_t u_\lambda \\
  &+ \frac12\pier{{}\Vert u_\lambda(t)\Vert_V^2} + \int_\Omega\widehat{\gamma_\lambda}(u_\lambda(t))
  + \frac{\lambda}2\int_\Omega|u_\lambda(t)|^2 \\
  &\quad =\frac12\pier{{}\Vert u_0\Vert_V^2} + \int_\Omega\widehat{\gamma_\lambda}(u_0) + \frac{\lambda}2\int_\Omega|u_0|^2
  +\int_{Q_t}\left(\pier{u_\lambda +{}} KT_\lambda \pier{( u_\lambda)} - g\right)\partial_t u_\lambda\\
  &\quad\leq \pier{{}\Vert u_0\Vert_V^2} + \int_\Omega\widehat{\gamma}(u_0)   +\frac12\int_{Q_t}|\partial_t u_\lambda|^2 + \pier{{}2\l|g\r|^2_{L^2(0,T; H)} + 2\big(1+ K^2\big)\int_0^t  \Vert u_\lambda(r)\Vert_V^2dr }\,.
  \end{split}
\]
Hence, 
thanks to the hypotheses \eqref{g}--\eqref{u_0}, the monotonicity of $\beta_\lambda$,
the positivity of $\widehat{\gamma_\lambda}$, \pier{the Gronwall lemma and the Poincar\'e inequality} we deduce that\pier{, for every $\lambda\in(0,1)$,
\begin{gather}
  \label{est1}
  \l|\mu_\lambda\r|_{L^2(0,T; \piecol{V_0})} \leq c\,,\\
  \label{est2}
  \pier{\l|u_\lambda\r|_{H^1(0,T; H)\cap L^\infty (0,T;V)}} \leq c\,,\\
  \label{est4}
  \l|\widehat{\gamma_\lambda}(u_\lambda)\r|_{L^\infty(0,T; L^1(\Omega))} \leq c \,.
\end{gather}%
}

\subsection{The second estimate}
The idea of the second estimate is to formally test \eqref{eq1_approx}
by $\partial_t \mu_\lambda$, the time derivative of \eqref{eq2_approx} by $\partial_t u_\lambda$
and then to take the difference: however, the regularities of $\mu_\lambda$ and $u_\lambda$ do not allow 
us to do so.

Let us give a rigorous argument. The first step is to identify the initial values 
of the approximated solutions\pier{%
$$\mu_{0,\lambda}:=\mu_\lambda(0) \quad \hbox{and}  \quad u_{0,\lambda}':=\partial_tu_\lambda(0). $$ 
Note that we already have $u_\lambda(0)=u_0$.}
The natural idea is to require that
$(\mu_{0,\lambda}, u_0, u_{0, \lambda}')$ satisfy the elliptic system induced by \eqref{eq1_approx}--\eqref{eq2_approx}
in $t=0$: in this way, we can show that $\mu_{0,\lambda}$ and $u_{0,\lambda}'$ are 
\pier{well defined} and unique.
We collect these remarks in the following lemma.
\begin{lem}
  \label{lemma4}
  There exists a unique \pier{pair $(\mu_{\lambda_0}, u_{\lambda, 0}')\in (V_0\cap H^2(\Omega))\times H$} such that
  \[
  \begin{cases}
    u_{0, \lambda}'-\Delta\mu_{0, \lambda}=0\,,\\
    \mu_{0, \lambda} = u_{0, \lambda}' + \beta_\lambda(u_{0,\lambda}') 
    - \Delta u_0 + \gamma_\lambda(u_0) + \lambda u_0+g(0) - KT_\lambda \pier{(u_0)}\,.
  \end{cases}
  \]
  Moreover, there exists a positive constant $C$, independent of $\lambda$, such that
  \beq
    \label{est_init}
    \l|\mu_{0,\lambda}\r|_{V_0}+ \l|u_{0,\lambda}'\r|_H + 
    \l|\widehat{\beta_\lambda^{-1}}\left(\beta_\lambda(u_{0,\lambda}')\right)\r|_{L^1(\Omega)}
    \leq C
  \eeq
  for every $\lambda\in(0,1)$, where \pier{$\widehat{\beta_{\lambda}^{-1}}:
  \Ar\rarr[0,+\infty)$ represents the convex conjugate function of $\widehat{\beta_{\lambda}}$, i.e., the Moreau\pier{--Yosida} regularization of $\widehat{\beta}$.}
\end{lem}
\begin{proof}
  The existence and uniqueness of $(\mu_{0,\lambda}, u_{0,\lambda})$ can be shown easily using the 
  hypotheses \eqref{g}--\eqref{u_0} as we did in the study of \eqref{system_init}. Moreover, testing the first equation
  by $\mu_{0,\lambda}$, the second by $u_{0,\lambda}'$ and taking the difference, thanks to \eqref{g}--\eqref{u_0} and the
  Young inequality we \pier{obtain}
  \[
    \begin{split}
    &\int_\Omega|\nabla\mu_{0,\lambda}|^2+\int_\Omega|u_{0,\lambda}'|^2+
    \int_\Omega\beta_\lambda(u_{0,\lambda}')u_{0,\lambda}'\\
    &\quad{}=\int_\Omega\left(\Delta u_0-\gamma_\lambda(u_0)-\lambda u_0-g(0)+KT_\lambda\pier{(u_0)}\right)u_{0,\lambda}'
    \leq \pier{c} + \frac12\int_\Omega|u_{0,\lambda}'|^2\,.
    \end{split}
  \]
\pier{Then,} in order to prove \eqref{est_init}, \pier{we start to observe that 
(cf., e.g., \cite{Barbu, Brezis}) that $\widehat{\beta_{\lambda}^{-1}}$ is a nonnegative function with minimum $0$ in $0$, such that its subdifferential $\partial\,\widehat{\beta_{\lambda}^{-1}}$ coincides with the inverse 
graph $\beta_{\lambda}^{-1}$  of the monotone and Lipschitz continuous function $\beta_\lambda $. Then, it is not difficult to check that}
  \[
  \widehat{\beta_\lambda^{-1}}\left(\beta_\lambda(u_{0,\lambda}')\right)\leq
  \widehat{\beta_\lambda^{-1}}\left(\beta_\lambda(u_{0,\lambda}')\right)+
  \pier{\widehat{\beta_\lambda}}\left(u_{0,\lambda}'\right) = \beta_\lambda(u_{0,\lambda}')u_{0,\lambda}'
  \]
  and the last quantity \pier{is nonnegative and has already been estimated in $L^1(\Omega)$}.
\end{proof}
\ele{Now, we proceed \pcol{with a somehow technical argumentation}, \pcol{by discretizing} the time interval and \pcol{rigorously performing} the estimate in the discrete time setting.} Fix $t\in[0,T]$: for any $n\in\En$, set $\tau_n:=\pier{t/n}$ and $t^i_n:=i\tau_n$
for $i=0,\ldots,n$. We know that \eqref{eq1_approx}--\eqref{eq2_approx} hold in $t^i_n$ for every $i=0,\ldots, n$.
Now, the idea is the following: we test
\eqref{eq1_approx} evaluated in $t^i_n$ by $\mu_\lambda(t^i_n)-\mu_\lambda(t_n^{i-1})$, 
\pier{as well as} the difference
between \eqref{eq2_approx} in $t^i_n$ and \eqref{eq2_approx} in $t^{i-1}_n$ by $\partial_t u_\lambda(t_n^i)$, 
and then we subtract the two quantities. We \pier{infer that}
\[
  \begin{split}
    \frac12&\int_\Omega|\nabla\mu_\lambda(t^i_n)|^2 - \frac12\int_\Omega|\nabla\mu_\lambda(t_n^{i-1})|^2
    +\frac12\int_\Omega|\nabla\mu_\lambda(t_n^i)-\nabla\mu_\lambda(t_n^{i-1})|^2\\
    &+\frac12\l|\partial_tu_\lambda(t_n^i)\r|^2_H - \frac12\l|\partial_tu_\lambda(t_n^{i-1})\r|^2_H
    +\frac12\l|\partial_tu_\lambda(t_n^i)-\partial_tu_\lambda(t_n^{i-1})\r|^2_H\\
    &+\int_\Omega\left(\beta_\lambda(\partial_t u_\lambda(t_n^i))-\beta_\lambda(\partial_t u_\lambda(t_n^{i-1}))\right)\partial_t u_\lambda(t_n^i)
    +\int_\Omega\left(\nabla u_\lambda(t_n^i)-\nabla u_\lambda(t_n^{i-1})\right)\cdot\nabla\partial_t u_\lambda(t_n^i)\\
    &+\int_\Omega\left(\gamma_\lambda(u_\lambda(t_n^i))-\gamma_\lambda(u_\lambda(t_n^{i-1}))\right)\partial_t u_\lambda(t_n^i)
    +\lambda\int_\Omega\left(u_\lambda(t_n^i)-u_\lambda(t_n^{i-1})\right)\partial_tu_\lambda(t_n^i)\\
    &=K\int_\Omega\left(T_\lambda \pier{(u_\lambda(t_n^i))}-T_\lambda \pier{(u_\lambda(t_n^{i-1}))}\right)\partial_t u_\lambda(t_n^i)
    -\int_\Omega\left(g(t_n^i)-g(t_n^{i-1})\right)\partial_t u_\lambda(t_n^i)\,.
  \end{split}
\]
\pier{Since $\partial\pier{\widehat{\beta_\lambda}}=\beta_\lambda$ and $\partial\widehat{\beta^{-1}_\lambda}=\beta^{-1}_\lambda$,
as a consequence} we have that
\[
  \left(\beta_\lambda(\partial_t u_\lambda(t_n^i))-\beta_\lambda(\partial_t u_\lambda(t_n^{i-1}))\right)\partial_t u_\lambda(t_n^i)
  \geq \widehat{\beta_\lambda^{-1}}\left(\beta_\lambda(\partial_t u_\lambda(t_n^i))\right)
  - \widehat{\beta_\lambda^{-1}}\left(\beta_\lambda(\partial_t u_\lambda(t_n^{i-1}))\right)\,.
\]
Hence, summing over $i=1,\ldots,n$ we obtain
  \begin{align*}
    &\frac12\int_\Omega|\nabla\mu_\lambda(t)|^2 +\frac12\l|\partial_tu_\lambda(t)\r|^2_H + 
    \int_\Omega\widehat{\beta_\lambda^{-1}}\left(\beta_\lambda(\partial_t u_\lambda(t))\right)\\
    &\quad{}+\sum_{i=1}^n\tau_n\int_\Omega\frac{\nabla u_\lambda(t_n^i)-\nabla u_\lambda(t_n^{i-1})}{\tau_n}\cdot\nabla\partial_t u_\lambda(t_n^i)\\
    &\quad{}+\sum_{i=1}^n\tau_n\int_\Omega\frac{\gamma_\lambda(u_\lambda(t_n^i))-\gamma_\lambda(u_\lambda(t_n^{i-1}))}{\tau_n}\partial_t u_\lambda(t_n^i)
    +\lambda\sum_{i=1}^n\tau_n\int_\Omega\frac{u_\lambda(t_n^i)-u_\lambda(t_n^{i-1})}{\tau_n}\partial_t u_\lambda(t_n^i)\\
    &\quad{}\leq \frac12\int_\Omega|\nabla\mu_{0,\lambda}|^2+\frac12\int_\Omega|u_{0,\lambda}'|^2+
    \int_\Omega\widehat{\beta_\lambda^{-1}}\left(\beta_\lambda(u_{0, \lambda}')\right)\\
    &\quad\quad{}+\sum_{i=1}^n\tau_n\int_\Omega\left(K\frac{T_\lambda \pier{(u_\lambda(t_n^i))} -T_\lambda \pier{(u_\lambda(t_n^{i-1}))}}{\tau_n}
    -\frac{g(t_n^i)-g(t_n^{i-1})}{\tau_n}\right)\partial_t u_\lambda(t_n^i)\,,
  \end{align*}
so that letting $n\rarr\infty$ we deduce that
\[
  \begin{split}
    \frac12&\int_\Omega|\nabla\mu_\lambda(t)|^2 +\frac12\l|\partial_tu_\lambda(t)\r|^2_H + 
    \int_\Omega\widehat{\beta_\lambda^{-1}}\left(\beta_\lambda(\partial_t u_\lambda(t))\right)\\
    &+\int_{Q_t}|\nabla\partial_t u_\lambda|^2+ \int_{Q_t}\gamma_\lambda'(u_\lambda)|\partial_tu_\lambda|^2
    +\lambda\int_{Q_t}|\partial_t u_\lambda|^2 \\
    &\leq\frac12\int_\Omega|\nabla\mu_{0,\lambda}|^2+\frac12\int_\Omega|u_{0,\lambda}'|^2+
    \int_\Omega\widehat{\beta_\lambda^{-1}}\left(\beta_\lambda(u_{0, \lambda}')\right)
    +\int_{Q_t}\left(\pier{K\, T'_\lambda (u_\lambda) \partial_t u_\lambda} -\partial_t g\right)\partial_t u_\lambda\,.
  \end{split}
\]
Now, since $|T_\lambda'|\leq 1$ and thanks to the estimate \eqref{est_init}, using the positivity of $\gamma_\lambda'$
and hypothesis \eqref{g} we \pier{infer} that
\[
  \begin{split}
  \frac12&\int_\Omega|\nabla\mu_\lambda(t)|^2 +\frac12\l|\partial_tu_\lambda(t)\r|^2_H + 
    \int_\Omega\widehat{\beta_\lambda^{-1}}\left(\beta_\lambda(\partial_t u_\lambda(t))\right)\\
    &{}+
    \int_{Q_t}|\nabla\partial_t u_\lambda|^2+\lambda\int_{Q_t}|\partial_t u_\lambda|^2
    \leq \pier{c}\left(1+\int_{Q_t}|\partial_t u_\lambda|^2\right).
  \end{split}
\]
Hence, by \piecol{\eqref{est2}} we deduce the following estimates\pier{, for every $\lambda \in (0,1)$,}
\begin{gather}
  \label{est6}
  \l|\mu_\lambda\r|_{L^\infty(0,T; \piecol{V_0})}\leq \pier{c}\,,\\
  \label{est7}
  \l|\partial_t u_\lambda\r|_{L^\infty(0,T; H)\cap L^2(0,T;V)} \leq \pier{c}\,,\\
  \label{est8}
  \l|\widehat{\beta_\lambda^{-1}}\left(\beta_\lambda(\partial_t u_\lambda)\right)\r|_{L^\infty(0,T; L^1(\Omega))} \leq \pier{c}\,.
\end{gather}
Finally, \pier{by virtue of \eqref{est7} and the elliptic regularity theory, 
from \eqref{eq1_approx} it follows that}
\begin{gather}
  \label{est7''}
  \l|\mu_\lambda\r|_{L^\infty(0,T; H^2(\Omega))\cap L^2(0,T; H^3(\Omega))}\leq \pier{c}\,.
\end{gather}

\subsection{\pcol{Further estimates} \ele{in the case of} growth hypothesis on $\psi''$}

\pier{Now, we work under the assumption \eqref{growth_psi}. We start proving an auxiliary result \pcol{which ensures that $|\gamma_\lambda'(r)| $ is uniformly bounded by a power function like the one in \eqref{growth_psi}} \ele{(\pcol{but} see the subsequent Remark \ref{potenzar}).}
\begin{lem}
  \label{lemma5}
Assume \eqref{psi1}--\eqref{psi5}, \eqref{gamma}--\eqref{gamma_hat}, \eqref{growth_psi} and, for every $\lambda \in (0,1)$, let $\gamma_\lambda$ be the Yosida approximation of $\gamma$. Then, there exists a constant $C$ independent of $\lambda$, such that 
\beq
\label{pier1}
|\gamma_\lambda' (r) | \leq C\big(1+ |r|^5\big) \quad \forall \, r\in \Ar .
\eeq    
\end{lem}
\begin{proof}
It is well known that (see, e.g., \cite{Barbu,Brezis}) for $r\in
\Ar$ we have $\gamma_\lambda (r) = (r- r_\lambda)/\lambda$, where 
$ r_\lambda $ uniquely solves the equation $r_\lambda + \lambda 
\gamma (  r_\lambda)= r$. Then, we observe that $\gamma_\lambda (r) = 
\gamma (  r_\lambda)$, whence differentiating and applying 
\eqref{gamma} and \eqref{growth_psi} yield 
$$|\gamma_\lambda' (r) |= |\gamma' (r_\lambda) | \leq M\big(1+ |r_\lambda |^5\big)+ K. $$ 
On the other hand, we multiply $r_\lambda + \lambda 
\gamma (  r_\lambda)= r$ by $r_\lambda - x_0 $, where $x_0$ in 
\eqref{gamma_hat} is such that $0 = \gamma (x_0) $. The, using the monotonicity of $\gamma$ we find out that 
$$ | r_\lambda |^2 \leq x_0  r_\lambda + |r| | r_\lambda | + 
|r| |x_0| .$$
Hence, with the help of Young inequality and rearranging, it is straightforward to
conclude that $ | r_\lambda | \leq c(1+ |r|)$, which entails \eqref{pier1}. 
\end{proof}
\luca{Next, we perform our estimate.} Owing to Lemma~\ref{approx_sol}, we can take the gradient of
\eqref{eq2_approx} tested by $\nabla\partial_t u_\lambda$, sum with 
\eqref{eq2_approx} tested by $\partial_t\gamma_\lambda(u_\lambda)$}
and integrate on $[0,t]$ for any given $t\in[0,T]$:
we obtain
\[
  \begin{split}
  \int_{Q_t}&|\nabla\partial_tu_\lambda|^2 +\luca{\int_{Q_t}\pier{\gamma'_\lambda}(u_\lambda)|\partial_tu_\lambda|^2}+
  \int_{Q_t}\beta_\lambda'(\partial_tu_\lambda)|\nabla\partial_tu_\lambda|^2+
  \luca{\int_{Q_t}\pier{\gamma'_\lambda (u_\lambda)
  \partial_t u_\lambda}\beta_\lambda(\partial_t u_\lambda)}\\
  &+\luca{\frac12\int_\Omega|\Delta u_\lambda(t)|^2
  \pier{{}-\int_{Q_t}\Delta u_\lambda \partial_t \gamma_\lambda(u_\lambda)
  +\int_{Q_t}\nabla\gamma_\lambda(u_\lambda)\cdot\nabla\partial_t u_\lambda  +   
  \frac12\int_\Omega|\gamma_\lambda(u_\lambda(t))|^2}}
    \\
  &= \frac12\luca{\int_\Omega|\Delta u_0|^2+\frac12\int_\Omega|\gamma_\lambda(u_0)|^2}
  +\int_{Q_t} \pier{\nabla\left(KT_\lambda(u_\lambda)\pier{{}- \lambda u_\lambda} +\mu_\lambda-g\right) }
  \cdot\nabla\partial_tu_\lambda \\
  &\quad{}+\luca{\int_{Q_t}\left(KT_\lambda(u_\lambda)\pier{{}- \lambda u_\lambda}
  +\mu_\lambda-g\right)  \partial_t\gamma_\lambda(u_\lambda)}\,.
  \end{split}
\]
\luca{Now, since $\partial_t u_\lambda\in L^2(0,T; V)$, we have $-\Delta \partial_tu_\lambda\in L^2(0,T; V^*)$,
so that integration by parts yields
\[
  \begin{split}
  \int_{Q_t}\nabla\gamma_\lambda(u_\lambda)\cdot\nabla\partial_t u_\lambda&=
  \int_0^t\left<-\Delta\partial_t u_\lambda(s), \gamma_\lambda(u_\lambda(s))\right>_{V^*,V}\,ds\\
  &=\int_{Q_t}\Delta u_\lambda \partial_t\gamma_\lambda(u_\lambda)
  -\int_\Omega\Delta u_\lambda(t)\gamma_\lambda(u_\lambda(t))+\int_\Omega\Delta u_0\gamma_\lambda(u_0)\,.
  \end{split}
\]
}\pier{Therefore, we can substitute in the previous relation and use the elementary fact that 
$$ \frac12 r^2 +  \frac12 s^2 -rs = \frac12 (r-s)^2 \quad \forall
\,r,s\in \Ar. $$
Then, by virtue of the monotonicity of $\gamma_\lambda$ and $\beta_\lambda$ along with the property $\beta_\lambda (0) = 0$ (cf.~\eqref{beta}), we deduce that 
\begin{align}
  \int_{Q_t}&|\nabla\partial_tu_\lambda|^2 
 +\frac12\int_\Omega|- \Delta u_\lambda(t)+ 
 \gamma_\lambda(u_\lambda(t))|^2 \nonumber \\
  &\leq \frac12\int_\Omega|- \Delta u_0 + \gamma_\lambda(u_0)|^2 
  +\int_{Q_t} \pier{\nabla \mu_\lambda }
  \cdot\nabla\partial_tu_\lambda \nonumber \\
  &\quad {}
 - \int_0^t\left<-\Delta\partial_t u_\lambda(s), 
 \left(KT_\lambda(u_\lambda)\pier{{}- \lambda u_\lambda}-g\right) (s)\right>_{V^*,V}\,ds \nonumber \\
  &\quad{}+\int_{Q_t}\mu_\lambda  \partial_t\gamma_\lambda(u_\lambda)  +\int_0^t\!\!\int_\Omega \partial_t\gamma_\lambda(u_\lambda(s))(KT_\lambda(u_\lambda)\pier{{}- \lambda u_\lambda}
 -g)(s)ds \,.  \label{pier2}
\end{align}
On account of \eqref{u_0}, \eqref{gamma},  \eqref{yos2} and \eqref{est6}--\eqref{est7}, we have that
\beq
\frac12\int_\Omega|- \Delta u_0 + \gamma_\lambda(u_0)|^2 
  +\int_{Q_t} \pier{\nabla \mu_\lambda }
  \cdot\nabla\partial_tu_\lambda \leq c. \label{pier3}
\eeq
On the other hand, we integrate by parts the third and fifth integral on the right-hand side of \eqref{pier2}. Thanks to 
\eqref{g}--\eqref{u_0}, the fact that $|T_\lambda'| \leq1$, the estimate \eqref{est2} and the Young inequality, we
obtain 
\begin{align}
 &- \int_0^t\left<-\Delta\partial_t u_\lambda(s), 
 \left(KT_\lambda(u_\lambda)\pier{{}- \lambda u_\lambda}-g\right) (s)\right>_{V^*,V}\,ds \nonumber \\
 &{}+\int_0^t\!\!\int_\Omega \partial_t\gamma_\lambda(u_\lambda(s))(KT_\lambda(u_\lambda)\pier{{}- \lambda u_\lambda}
 -g)(s)ds \nonumber \\
 &{}= \int_\Omega  (- \Delta u_\lambda(t) +  \gamma_\lambda(u_\lambda(t)))(KT_\lambda(u_\lambda)\pier{{}- \lambda u_\lambda} -g)(t)
 \nonumber \\
 &{}\quad - \int_\Omega  (- \Delta u_0 + \gamma_\lambda(u_0))(KT_\lambda(u_0)\pier{{}- \lambda u_0} -g(0))
 \nonumber \\
 &{}\quad  
 - \int_0^t\!\!\int_\Omega  (- \Delta u_\lambda +  \gamma_\lambda(u_\lambda))( KT'_\lambda(u_\lambda) \partial_t u_\lambda
 - \lambda \partial_t u_\lambda  -\partial_t g)
  \nonumber \\
 &{}\leq  \frac14\int_\Omega|- \Delta u_\lambda(t) +  \gamma_\lambda(u_\lambda(t))|^2 
+ \frac14  \int_0^t\!\!\int_\Omega  |- \Delta u_\lambda +  \gamma_\lambda(u_\lambda))|^2 + c. \label{pier4}
\end{align}
Finally, since $H^2(\Omega)\embed L^\infty(\Omega)$ and $V\embed L^6(\Omega)$ with continuous embeddings,
using the H\"older inequality we infer that
\ele{\begin{align}
  \int_{Q_t}\mu_\lambda\partial_t\gamma_\lambda(u_\lambda)&\leq
  \int_0^t\l|\mu_\lambda(s)\r|_{L^\infty(\Omega)}
  \l|\gamma_\lambda'(u_\lambda(s))\r|_{L^{6/5}(\Omega)}\l|\partial_tu_\lambda(s)\r|_{L^6(\Omega)}\,ds  \nonumber \\
  &\leq C\l|\mu_\lambda\r|_{L^\infty(0,T; H^2(\Omega))}
  \l|\gamma_\lambda'(u_\lambda)\r|_{L^2(0,T;L^{6/5}(\Omega))}
  \l|\partial_t u_\lambda\r|_{L^2(0,T; V)}. \label{Ele1}
\end{align}
Hence,} as Lemma~\ref{lemma5} implies that $|\gamma'(u_\lambda)|^{6/5}\leq c \big( 1+|u_\lambda|^6 \big) , $
in view of  the estimates \pier{\eqref{est2}, \eqref{est7} and \eqref{est7''}} we deduce a global bound independent of $\lambda$:
\beq
  \int_{Q_t}\mu_\lambda\partial_t\gamma_\lambda(u_\lambda) \leq c.
\label{pier5} 
\eeq
At this point, combining \eqref{pier2} with \eqref{pier3}--\eqref{pier5}, then applying the Gronwall lemma we find out that 
\beq 
 \| - \Delta u_\lambda +  \gamma_\lambda(u_\lambda))\|_{L^\infty(0,T; H)} \leq c. \label {pier6}
\eeq 
Now, a standard procedure (test of $- \Delta u_\lambda +  \gamma_\lambda(u_\lambda))=:f_\lambda $ by $- \Delta u_\lambda $
and integration by parts in the integral with $\gamma_\lambda(u_\lambda)$ by exploiting the monotonicity of $\gamma_\lambda$) leads us to the estimates 
\begin{gather}
  \label{est11}
  \l|\Delta u_\lambda\r|_{L^\infty(0,T;H)}\leq c\,,\\
  \label{est12}
  \l|\gamma_\lambda(u_\lambda)\r|_{L^\infty(0,T; H)}\leq c\, , 
\end{gather}
whence (cf.~\eqref{est2}) 
\beq 
 \|  u_\lambda \|_{L^\infty(0,T; W)} \leq c \label{pier7}
\eeq 
and, by comparison in \eqref{eq2_approx}, also that
\beq
  \label{est14}
   \l|\beta_\lambda(\partial_t u_\lambda)\r|_{L^\infty(0,T; H)}\leq c\,.
\eeq}
\ele{%
\begin{rmk}\label{potenzar}
Let us point out that we can get \eqref{pier5} after proving \eqref{pier1}. In particular, we are able to estimate $|\gamma_\lambda'(u_\lambda)|^{6/5}$ on the right hand side of \eqref{Ele1} in terms of 
$|u_\lambda|^6$ and then \pcol{obtain} the required estimate due to the fact that $u_\lambda$ \pcol{is bounded in $L^\infty(0,T;L^6(\Omega))$ (cf.~\eqref{est2}).} 
It is clear that this last result is related to the continuous embedding  $\pcol{V}\subset L^p(\Omega)$ for $1\leq p\leq 6$ holding in dimension 3. In the case of $\Omega\subseteq \Ar^2$, \pcol{\eqref{est2}} provides a bound of $\|u_\lambda\|_{L^\infty(0,T;L^p(\Omega))}$ for any $p<+\infty$, and thus we could \pcol{weaken} \eqref{pier1} and consequently also~\eqref{growth_psi}.                                                                                           
\end{rmk}%
}

\subsection{\pcol{Further estimates} \ele{in the case of} growth hypothesis on $\beta$}

\gius{In this subsection we drop the growth assumption on \pcol{$\psi''$} stated in \eqref{growth_psi} and replace it with the hypothesis that}{} $\beta$ is sublinear \pier{(cf.~\eqref{beta_sublinear})}\gius{.}{} \gius{We are going to show that in this case it is still possible to obtain the estimates \eqref{est11}-\eqref{est14}.}{}

\gius{As a start, we notice that the uniform bound \eqref{est14}}{, it} follows \gius{from \eqref{beta_sublinear},}{} \eqref{yos1}, and \eqref{est7}\gius{.}{that
\beq
  \label{est15}
  \l|\beta_\lambda(\partial_t u_\lambda)\r|_{L^\infty(0,T; H)}\leq c\, .
\eeq}
Hence, testing \eqref{eq2_approx} by $-\Delta u_\lambda (t) $ for any given 
$t\in[0,T]$ we infer that
\[
  \begin{split}
   &\int_\Omega|\Delta u_\lambda(t)|^2 +\int_{\Omega}\gamma'_\lambda(u_\lambda (t))|\nabla u_\lambda (t) |^2+\lambda\int_{\Omega}|\nabla u_\lambda (t) |^2\\
    &= \int_{\Omega}(\partial_t   u_\lambda + \beta_\lambda(\partial_t u_\lambda) +g - KT_\lambda \pier{(u_\lambda)} - \mu_\lambda)(t) \Delta u_\lambda(t)\,,
  \end{split}
\]
from which, using the Young inequality, \eqref{g} and the estimates \eqref{est2}, \eqref{est6}, \eqref{est7} we deduce \gius{\eqref{est11}.}{that 
\begin{gather}
  \label{est16}
  \l|\Delta u_\lambda\r|_{L^\infty (0,T; H)}\leq c\,. 
\end{gather}}
Then, by comparison in \eqref{eq2_approx}, we \gius{arrive at \eqref{est12}.}{have that
\beq
  \label{est18}
  \l|\gamma_\lambda(u_\lambda)\r|_{L^\infty (0,T; H)}\leq c\,.
\eeq }
We also point out that \eqref{est2}, \gius{\eqref{est11}}{\eqref{est16}} 
and the elliptic regularity theory imply \gius{\eqref{pier7}.}{that
\beq 
 \|  u_\lambda \|_{L^\infty(0,T; W)} \leq c \label{pier8}.
\eeq%
}
 
\gius{}{
\subsection{The fourth estimate: \ele{in the case of} growth hypothesis on $\beta$}

Since $\beta$ is sublinear \pier{(cf.~\eqref{beta_sublinear}), it follows from \eqref{yos1} and \eqref{est7} that
\beq
  \label{est15}
  \l|\beta_\lambda(\partial_t u_\lambda)\r|_{L^\infty(0,T; H)}\leq c\, .
\eeq
Hence, testing \eqref{eq2_approx} by $-\Delta u_\lambda (t) $ for any given $t\in[0,T]$ we infer that
\[
  \begin{split}
   &\int_\Omega|\Delta u_\lambda(t)|^2 +\int_{\Omega}\gamma'_\lambda(u_\lambda (t))|\nabla u_\lambda (t) |^2+\lambda\int_{\Omega}|\nabla u_\lambda (t) |^2\\
    &= \int_{\Omega }(\partial_t   u_\lambda + \beta_\lambda(\partial_t u_\lambda) +g - KT_\lambda \pier{(u_\lambda)} - \mu_\lambda)(t) \Delta u_\lambda(t)\,,
  \end{split}
\]
from which, using the Young inequality, \eqref{g} and the estimates \eqref{est2}, \eqref{est6}, \eqref{est7} we deduce that 
\begin{gather}
  \label{est16}
  \l|\Delta u_\lambda\r|_{L^\infty (0,T; H)}\leq c\,. 
\end{gather}
Then, by comparison in \eqref{eq2_approx}, we have that
\beq
  \label{est18}
  \l|\gamma_\lambda(u_\lambda)\r|_{L^\infty (0,T; H)}\leq c\,.
\eeq
We also point out that \eqref{est2}, \eqref{est16} and the elliptic regularity theory imply that
\beq 
 \|  u_\lambda \|_{L^\infty(0,T; W)} \leq c \label{pier8}.
\eeq%
} 
}


\section{The passage to the limit}
\setcounter{equation}{0}
\label{limit}

In this section we pass to the limit as $\lambda\searrow0$ in the approximated problem and obtain a solution
\pier{$(u,\mu,\xi)$  to \eqref{u}--\eqref{3}}.

Owing to the estimates \eqref{est1}--\eqref{est2}, \eqref{est6}--\eqref{est7''}, \eqref{est11}--\eqref{est14},
we infer that there exist
\begin{gather}
  \label{u_lim}
  u\in \pier{W^{1,\infty}(0,T; H)\cap H^1(0,T; V)\cap L^\infty(0,T; \pier{W})}\,,\\
  \label{mu_lim}
  \mu\in L^\infty(0,T; V_0\cap H^2(\Omega))\cap L^2(0,T; H^3(\Omega))\,,\\
  \label{xi_eta_lim}
  \xi\in L^\infty(0,T; H)\,, \qquad \eta\in L^\infty(0,T; H)
\end{gather}
such that the following convergences hold as $\lambda\searrow0$
(along a subsequence, which we still denote by $\lambda$
for simplicity):
\begin{gather}
  \label{conv1}
  u_\lambda\weakstar u \quad\text{in } \pier{W^{1,\infty}(0,T; H)  \cap L^\infty(0,T; W)}\,,\qquad
  u_\lambda\rarrw u \quad\text{in } H^1(0,T; V)\,,\\
  \label{conv2}
  \mu_\lambda\weakstar\mu \quad\text{in } L^\infty(0,T; V_0\cap H^2(\Omega))\,, \qquad
  \mu_\lambda\rarrw\mu \quad\text{in } L^2(0,T; H^3(\Omega))\,,\\
  \label{conv3}
  \beta_\lambda(\partial_t u_\lambda)\weakstar\xi \quad\text{in } L^\infty(0,T; H)\,,\\
  \label{conv4}
  \gamma_\lambda(u_\lambda) \weakstar\eta \quad\text{in } L^\infty(0,T; H)\,,\\
  \label{conv5}
  \lambda u_\lambda\rarr0 \quad\text{in } \pier{W^{1,\infty}(0,T; H)\cap H^1(0,T; V)\cap L^\infty(0,T; \pier{W})}\,.
\end{gather}
\pier{From \eqref{conv1} and a well-known compactness result (see \cite[Cor.~4, p.~85]{simon}) it follows that
\beq
  \label{conv6}
  u_\lambda\rarr u \quad\text{in } C^0([0,T]; V \cap C^0 (\,\overline{\Omega}\,)) 
\eeq
and consequently in $C^0 \left(\overline{Q}\right)$, of course. 
Moreover, recalling \eqref{trunc} it is a standard matter to check that 
\beq
  \label{conv7}
  T_\lambda \pier{(u_\lambda)} \rarr u \quad\text{in } C^0 \left(\overline{Q}\right) \,.
\eeq
Now,} passing to the weak limit in \eqref{eq1_approx}--\eqref{eq2_approx}, we deduce that 
\beq
  \label{lim}
 \pier{ \partial_t u - \Delta\mu=0\,,\quad
  \mu=\partial_t u + \xi - \Delta u + \eta + g - Ku \quad \hbox{a.e. in } \, (0,T). } 
\eeq
Furthermore, the convergences \eqref{conv4} and \eqref{conv6} together with the 
strong-weak closure of the maximal monotone operator $\gamma$ ensure that $\eta=\gamma(u)$.
Consequently, recalling the definition \eqref{gamma} of $\gamma$, 
the second equation \pier{in \eqref{lim} becomes $\mu=\partial_t u + \xi - \Delta u + \psi'(u) + g$ almost everywhere in $(0,T)$.
Moreover, the initial condition \eqref{3} follows from \eqref{approx2} and \eqref{conv6}.}

\pier{The last thing that we have to prove is that $\xi\in\beta(\partial_t u)$ almost everywhere in $Q$. To this aim,
as we did at the beginning of Section \ref{estimates}, we test \eqref{eq1_approx} by $\mu_\lambda$,
\eqref{eq2_approx} by $\partial_t u_\lambda$, take the difference and integrate on $[0,T]$:
using the fact that $\widehat{\gamma_\lambda}\leq \widehat{\gamma}$ \pier{for all $\lambda \in (0,1)$} we obtain
\begin{align}
  & \int_Q|\nabla\mu_\lambda|^2+ \int_Q|\partial_t u_\lambda|^2 +
  \int_Q\beta_\lambda(\partial_t u_\lambda)\partial_tu_\lambda +\frac12\int_\Omega|\nabla u_\lambda(T)|^2 + \int_\Omega\widehat{\gamma_\lambda}(u_\lambda(T))   \nonumber \\
  &=\frac12\int_\Omega|\nabla u_0|^2 + \int_\Omega\widehat{\gamma_\lambda}(u_0) + \int_Q\left(KT_\lambda \pier{(u_\lambda)}- \lambda u_\lambda  - g\right)\partial_t u_\lambda \nonumber \\
  &\leq\frac12\int_\Omega|\nabla u_0|^2 + \int_\Omega\widehat{\gamma}(u_0)+ \int_Q\left(KT_\lambda \pier{(u_\lambda)}- \lambda u_\lambda -g\right)\partial_t u_\lambda\,. \label{pier9}
\end{align}
Now, we claim that
\beq 
\int_\Omega\widehat{\gamma_\lambda}(u_\lambda(T)) \to \int_\Omega\widehat{\gamma}(u(T))  \quad \hbox{as }  \, 
\lambda \searrow 0 .  \label{pier10}
\eeq
Indeed, this is a consequence of \eqref{conv6} and \cite[Prop~2.11, p.~39]{Brezis} provided  that  $ \widehat{\gamma}(u(T))
\in L^1(\Omega)$: now, since $u\in H^1(0,T;H)$ and $\eta=\gamma(u) \in L^2(0,T;H) $ it turns out that (see~\cite[Lemme~3.3, p.~73]{Brezis})
\begin{align*} 
&\hbox{the function } \  t\mapsto \int_\Omega  \widehat{\gamma}(u(t)) \ \hbox{ is absolutely continuous in $[0,T]$} \\
&\hbox{and its derivative equals }  \ (\eta(t) , \partial_t u(t) ) \ \hbox{ for a.e. } t\in (0,T),
\end {align*}
whence \eqref{pier10} follows}

\pier{Owing to \eqref{conv1}--\eqref{conv7}, the weak lower semicontinuity of the norms and \eqref{pier10},
from \eqref{pier9} we deduce that 
\begin{align}
  &\limsup_{\lambda\searrow0}\int_Q\beta_\lambda(\partial_t u_\lambda)\partial_t u_\lambda \nonumber \\
  &\leq
  \frac12\int_\Omega|\nabla u_0|^2 + \int_\Omega\widehat{\gamma}(u_0) + \int_Q\left(Ku-g\right)\partial_t u\nonumber\\
  &\quad -\liminf_{\lambda\searrow0}\left[\int_Q|\nabla\mu_\lambda|^2+ \int_Q|\partial_t u_\lambda|^2+
  \frac12\int_\Omega|\nabla u_{\lambda}(T)|^2+\int_\Omega\widehat{\gamma_\lambda}(u_\lambda(T))\right]\nonumber\\
  &\leq \frac12\int_\Omega|\nabla u_0|^2 + \int_\Omega\widehat{\gamma}(u_0) + \int_Q\left(Ku-g\right)\partial_t u 
  \nonumber\\
  &\quad -\int_Q|\nabla\mu|^2 -\int_Q|\partial_t u|^2-\frac12\int_\Omega|\nabla u(T)|^2
  -\int_\Omega\widehat{\gamma}(u(T))\,. \label{pier11}
\end{align}
Now, if we test the first equation of \eqref{lim} by $\mu$, the second by $\partial_t u$, take the difference and integrate
on $[0,T]$, in a similar way as before we arrive at
\begin{align}
  &\int_Q|\nabla\mu|^2+ \int_Q|\partial_t u|^2 +
  \int_Q\xi\, \partial_tu
  +\frac12\int_\Omega|\nabla u(T)|^2 + \int_\Omega\widehat{\gamma}(u(T))\nonumber \\
  &=\frac12\int_\Omega|\nabla u_0|^2 + \int_\Omega\widehat{\gamma}(u_0)
  + \int_Q\left(K u-g\right)\partial_t u\,. \label{pier12}
\end{align}
Therefore, \eqref{pier11} and \eqref{pier12} allow us to infer that
\[
  \limsup_{\lambda\searrow0}\int_Q\beta_\lambda(\partial_t u_\lambda)\partial_t u_\lambda \leq
  \int_Q\xi\, \partial_tu\,.
\]
This condition together with the weak convergences \eqref{conv1} and \eqref{conv3}
ensures that $\xi\in\beta(\partial_t u)$ almost everywhere in $Q$, thanks to a standard result for 
maximal monotone operators (see, e.g., \cite[Prop.~1.1, p.~42]{Barbu}).  By this, we conclude the proofs of 
Theorems~\ref{thm1} and~\ref{thm1bis}.}


\section{Continuous dependence and uniqueness}
\label{cont_dep}

\pier{In this section we prove Theorem~\ref{contdep}. Let us write \eqref{incl}--\eqref{3} for both the sets of 
data $(g_i, u_{0,i} )$ and the corresponding solutions  $(u_i, \mu_i, \xi_i)$, $i=1,2$; then take the differences. 
Let us point out that, within this section, we use the notations $u=u_1 - u_2 $,  $\mu=\mu_1 - \mu_2 $,  $\xi =\xi_1 - \xi_2 $
and  $g=g_1 - g_2 $,  $u_0=u_{0,1} - u_{0,2} $. }

\pier{Now, in the difference of equations \eqref{2} we add and subtract $u$; then, we test the difference of 
equations~\eqref{1} by $\mu $  and  the difference of equations \eqref{2}  by $\partial_t u$; subsequently, we
subtract and integrate with respect to $t$. It is straightforward to obtain 
\begin{align}
  &\int_{Q_t}|\nabla\mu|^2+ \int_{Q_t}|\partial_t u|^2 +
  \int_{Q_t} \xi\, \partial_tu
  +\frac12\|  u(t )\|_V^2 \nonumber \\
  &=\frac12\|  u_0 \|_V^2 
  + \int_{Q_t}   \left(u_1 -  \psi'(u_1) -u_2 +  \psi'(u_2) \right)\partial_t u
  - \int_{Q_t}   g\, \partial_t u
   \label{pier13}
\end{align}
for all $t\in [0,T].$  Hence, in view of \eqref{psi2} and \eqref{hyp-pier}, since the function $r \mapsto r - \psi'(r)$ is Lipschitz continuous in $\big[\, \overline{a}, \overline{b}\, \big]$, by the Young inequality  we infer that  
\beq 
  \int_{Q_t}   \left(u_1 -  \psi'(u_1) -u_2 +  \psi'(u_2) \right)\partial_t u \leq \frac14 \int_{Q_t}|\partial_t u|^2 + c 
  \int_0^t \|  u(s )\|_H^2 ds .
    \label{pier14}
\eeq
On the other hand, we have that 
\beq 
 - \int_{Q_t}   g\, \partial_t u\leq \frac14 \int_{Q_t}|\partial_t u|^2 +   \int_0^t \|  g(s )\|_H^2 ds .
    \label{pier15}
\eeq
Thus, combining \eqref{pier13} with \eqref{pier14}--\eqref{pier15} and observing that 
$$
\int_{Q_t}|\nabla\mu|^2 \ \hbox{ yields a control of } \ \int_0^t \|  \mu(s )\|_{V_0}^2 ds ,
$$
applying the Gronwall lemma leads us to finally obtain \eqref{dipcont}.}


%


\end{document}